\documentclass[11pt,reqno]{amsart}
\usepackage{amsmath, amsfonts, amsthm, amssymb,mathrsfs, amscd, graphicx, cite, color}

\usepackage{float}
\usepackage{epsfig}

\title[Vanishing viscosity limit to planar rarefaction wave for 2D Navier-Stokes]
{Vanishing viscosity limit to the planar rarefaction wave for the two-dimensional compressible Navier-Stokes equations}

\author[L.-A. Li]{Lin-An Li}
\address{Institute of Applied Mathematics, AMSS, Academia Sinica, Beijing 100190, P. R. China; 
  and School of Mathematical Sciences, University of Chinese Academy of Sciences, Beijing 100049, P. R. China.}
\email{linanli@amss.ac.cn}

\author[D. Wang]{Dehua Wang}
\address{Department of Mathematics, University of Pittsburgh, Pittsburgh, PA 15260, USA.}
\email{dwang@math.pitt.edu}

\author[Y. Wang]{Yi Wang}
\address{
Academy of Mathematics and Systems Science, Chinese Academy of Sciences, Beijing 100190, P. R. China;
 and School of Mathematical Sciences, University of Chinese Academy of Sciences, Beijing 100049, P. R. China.}
\email{wangyi@amss.ac.cn}

\newtheorem{theorem}{Theorem}[section]
\newtheorem{lemma}{Lemma}[section]

\newtheorem{proposition}{Proposition}[section]

\theoremstyle{definition}

\theoremstyle{remark}
\newtheorem{remark}{Remark}[section]

\newcommand{\bbr}{\mathbb R}

\newcommand{\bbt}{\mathbb T}
\newcommand{\bbz}{\mathbb Z}

\renewcommand{\div}{{\rm div}}

\def\charf {\mbox{{\text 1}\kern-.30em {\text l}}}

\def\lam{\lambda}  

\def\di{\displaystyle}

\begin{document}

\date{\today}

\begin{abstract}
The vanishing viscosity limit of the two-dimensional (2D) compressible isentropic Navier-Stokes equations is studied in the case that the corresponding 2D inviscid Euler equations admit a planar rarefaction wave solution. It is proved that there exists a family of smooth solutions  for the 2D compressible Navier-Stokes equations  converging to the planar rarefaction wave solution with arbitrary strength for the 2D Euler equations. A uniform convergence rate is obtained in terms of the viscosity coefficients away from the initial time. In the proof, the hyperbolic wave is crucially introduced to recover  the physical viscosities of the inviscid rarefaction wave profile,   in order to rigorously justify the vanishing viscosity limit.
\end{abstract}

\date{\today}

\maketitle 
%
%
\section{Introduction}  
\setcounter{equation}{0}
In this paper, we investigate the vanishing viscosity limit of the two-dimensional  compressible and isentropic Navier-Stokes equations:
\begin{equation}  \label{NS}
\begin{cases}
\displaystyle \rho_t + \div(\rho \textbf{u}) = 0,    \\ 
\displaystyle (\rho \textbf{u})_t + \div(\rho \textbf{u} \otimes \textbf{u}) + \nabla p(\rho) = \mu_1\triangle\textbf{u} + (\mu_1 + \lam_1)\nabla \div\textbf{u},
\end{cases}
\end{equation}
where $\rho = \rho(t, x_1, x_2) \geq 0, \textbf{u} = \textbf{u}(t, x_1, x_2) = (u_1, u_2)(t, x_1, x_2)$ and $p=p(t, x_1, x_2)$ represent the fluid density, velocity and  pressure, respectively;   $(x_1, x_2) \in \bbr^2$ is the spatial  variable   and $t>0$ 
is the time variable.
The pressure $p=p(\rho)$ is given by the   $\gamma$-law:
\[
p(\rho) = \frac{\rho^\gamma}{\gamma}
\]
with $\gamma \geq 1$   the adiabatic constant.
Both the shear viscosity $\mu_1$ and the bulk viscosity $\lam_1$ are constant satisfying 
\begin{equation} \label{VIS}
  \mu_1 > 0, \qquad \mu_1 + \lam_1 \geq 0,
\end{equation}
and we take
$$
\mu_1 = \mu\varepsilon,\qquad \lam_1 = \lam\varepsilon,
$$ 
where  $\varepsilon>0$ is  the vanishing parameter,  and $\mu$ and $\lam$ are the prescribed uniform-in-$\varepsilon$ constants. For the spatial domain, we consider the case $x_1\in \bbr$   and $x_2\in \bbt:=\bbr/\mathbb{Z}$, the one-dimensional unit flat torus.

Since we are concerned with the vanishing viscosity limit to the planar rarefaction wave for the system \eqref{NS}, we consider the following initial data:
\begin{equation} \label{NSI}
(\rho, \textbf{u})(0, x_1, x_2) = (\rho, u_1, u_2)(0, x_1, x_2) = (\rho_0(x_1,x_2), u_{10}(x_1, x_2), u_{20}(x_1, x_2)),
\end{equation}
 and the far field  condition  of solutions in the $x_1$-direction:
\begin{equation}
 (\rho, u_1, u_2)(t, x_1, x_2)\rightarrow (\rho_\pm, u_{1\pm}, 0), \qquad{\rm as} \quad x_1 \rightarrow \pm\infty,
\end{equation}
where $\rho_\pm>0, u_{1\pm}$ are the prescribed constants.
The periodic boundary condition is imposed on $x_2 \in\mathbb{T}$ for the solution $(\rho, u_1, u_2)(t,x_1,x_2)$ to \eqref{NS}, where the end states $(\rho_\pm, u_{1\pm})$ are connected by the rarefaction wave solution to the Riemann problem of  the corresponding one-dimensional (1D) hyperbolic system of conservation laws:
\begin{equation}  \label{ES}
\begin{cases}
\displaystyle \rho_t + (\rho u_1)_{x_1} = 0,               \quad\qquad \qquad x_1 \in \bbr, ~t > 0, \\
\displaystyle (\rho u_1)_t + (\rho u_1^2 + p(\rho))_{x_1} = 0,
\end{cases}
\end{equation}
with the Riemann initial data
\begin{equation} \label{ini}
(\rho_0^r, u_{10}^r)(x_1) = \begin{cases}
(\rho_-, u_{1-}),     \quad x_1 < 0,\\
(\rho_+, u_{1+}),     \quad x_1 > 0.
\end{cases}
\end{equation}
Formally speaking, as $\varepsilon$ tends to zero, the two-dimensional (2D) compressible Navier-Stokes equations \eqref{NS}-\eqref{NSI} converge to the corresponding 2D inviscid compressible Euler equations:
\begin{equation}  \label{2ES}
\begin{cases}
\displaystyle \rho_t + (\rho u_1)_{x_1} + (\rho u_2)_{x_2} = 0,               \quad\qquad \qquad (x_1, x_2) \in \bbr\times\bbt, ~t > 0, \\
\displaystyle (\rho u_1)_t + (\rho u_1^2 + p(\rho))_{x_1} + (\rho u_1 u_2)_{x_2} = 0, \\
\di (\rho u_2)_t + (\rho u_1u_2 )_{x_1} + (\rho u_2^2 + p(\rho))_{x_2} = 0.
\end{cases}
\end{equation}
In the regime of the planar rarefaction wave, we consider the Euler system \eqref{2ES} with the following Riemann initial data
\begin{equation} \label{2ini}
(\rho_0^r, u_{10}^r, u_{20}^r)(x_1) = \begin{cases}
(\rho_-, u_{1-}, 0),     \quad x_1 < 0,\\
(\rho_+, u_{1+}, 0),     \quad x_1 > 0.
\end{cases}
\end{equation}
We note that, although  the $u_2$-component is continuous on the both sides of $x_1 = 0$ in \eqref{2ini},
 the one-dimensional Riemann problem \eqref{ES}-\eqref{ini} and the two-dimensional Riemann problem \eqref{2ES}-\eqref{2ini} have some substantial difference. For example, the results in \cite{CDK,CK} indicate that 
 there are infinitely many bounded admissible weak solutions to \eqref{2ES}-\eqref{2ini} satisfying the  entropy condition for the shock Riemann initial data,  
 and their   construction of weak solutions  based on the convex integration method in DeLellis and Szekelyhidi \cite{DS}   for the two-dimensional system  may not  be applied to the one-dimensional problem \eqref{ES}-\eqref{ini}. 
 The results in \cite{CK, CDK} were extended to the Riemann initial data with shock or contact discontinuity in \cite{KM,BCK}. 
Nevertheless,  the uniqueness of the uniformly bounded admissible weak solution was proved in Chen-Chen\cite{CC},  Feireisl-Kreml \cite{FK},  and Feireisl-Kreml-Vasseur \cite{FKV} for the Riemann solution  containing only rarefaction waves to \eqref{2ES}-\eqref{2ini} 
even with vacuum states, which is similar to the one-dimensional case.
Our current paper is devoted to establish the mathematical justification of the vanishing viscosity limit of the 2D compressible Navier-Stokes equations \eqref{NS}-\eqref{NSI} to the planar rarefaction wave solution of the 2D Riemann problem of the corresponding compressible Euler equations \eqref{2ES}-\eqref{2ini}.
  
There have been many results in literature on the vanishing viscosity limit to the basic wave patterns for the system of viscous conservation laws in the one-dimensional case. For the 1D system of the hyperbolic conservation laws with artificial viscosity,
Goodman and Xin \cite{GX} applied a matched asymptotic expansion method to first prove the viscous limit for the piecewise smooth solutions separated by noninteracting shock waves.
 Later Yu \cite{Y} extended the result in \cite{GX} for the corresponding hyperbolic conservation laws with both shocks and initial layers.     Bianchini and Bressan \cite{BB} proved the vanishing artificial viscosity limit in the general small BV spaces although the problem is still unsolved for the physical systems such as the compressible Navier-Stokes equations. 
 For the one-dimensional compressible   isentropic Navier-Stokes equations, the vanishing viscosity limit was  obtained in Hoff and Liu \cite{HL}    for the piecewise constant shocks even with initial layers,      in Xin \cite{X-1} for the zero dissipation limit to the rarefaction wave for both the Riemann data and the well-prepared smooth data,  and in Huang, Li and Wang \cite{H-L-W} and Li and Wang \cite{L-W}  for the  zero dissipation limit in the case of the rarefaction wave connected with the vacuum states. 
The result in \cite{GX} was extended in Wang \cite{W-1} to the one-dimensional isentropic Navier-Stokes equations.
For the nonisentropic Navier-Stokes equations, the results on the zero dissipation limit to the corresponding full Euler system with basic wave patterns can be found in  \cite{JNS,XZ}  for the rarefaction wave, in \cite{W-2} for the shock wave, in  \cite{M} for the contact discontinuity, and in \cite{HWY-1, HWY-2} for the superposition of two rarefaction waves and a contact discontinuity and the superposition of one shock and one rarefaction wave cases. 
More recently, Huang, Wang, Wang and Yang \cite{HWWY} justified the vanishing viscosity limit of the compressible Navier-Stokes equations for the generic 1D Riemann solution which may contain shock and rarefaction waves and contact discontinuity. On the other hand,  Chen and Perepelitsa \cite{CP} proved the vanishing viscosity limit to the compressible Euler equations for the one-dimensional compressible Navier-Stokes equations in $L^p$-framework by using the compensated compactness method. 
For other related results on the inviscid limit in literature, see \cite{CG18a,CLQ18,CEIV,DN2018,Kato84,BTW18,Masmoudi2007} and the references therein.

Although there have been satisfactory results mentioned above on the vanishing viscosity limit to the basic wave patterns for the viscous conservation laws in the one-dimensional case, there are very few results on the vanishing viscosity limit to the planar wave patterns for the compressible Navier-Stokes equations \eqref{NS} in the multi-dimensional case.  Motivated by the recent progress on the time-asymptotic stability of the planar rarefaction wave to the multi-dimensional compressible Navier-Stokes equations by Li, Wang and Wang \cite{LWW} and Li and Wang \cite{LW}, in the present paper we aim to justify the vanishing viscosity limit to the planar rarefaction wave for the two-dimensional compressible Navier-Stokes equations \eqref{NS} with physical constraints \eqref{VIS} and obtain the decay rate with respect to the viscosity coefficients.
 Compared with the one-dimensional vanishing viscosity limit results in \cite{X-1, H-L, H-L-W}, the additional difficulties here lie in the propagation of the planar rarefaction wave in $x_2$-direction and its interactions with the wave in $x_1$-direction due to the higher dimensionality. Therefore, we need to introduce a new wave, called hyperbolic wave, to recover the physical viscosity of the compressible Navier-Stokes equations for the inviscid rarefaction wave profile satisfying the compressible Euler equations exactly, which is partially motivated by the work \cite{HWY-2} for the viscous limit of the one-dimensional  full compressible Navier-Stokes equations in the case of superposition of both shock and rarefaction waves. Note that this hyperbolic wave plays a crucial role for the uniform estimates with respect to the viscosity coefficients for the perturbation of the solution to \eqref{NS} around both the rarefaction wave profile and the new hyperbolic wave and it seems that we can not justify the vanishing viscosity limit for 2D Navier-Stokes equations \eqref{NS} without this hyperbolic wave by  using only the rarefaction wave profile itself. By using the rarefaction wave profile and the new hyperbolic wave as the ansatz, the vanishing viscosity limit problem can be reformulated as a time-asymptotic stability problem around the background solution profile which consists of rarefaction wave and hyperbolic wave so that the energy method can be applied after some suitable scalings. Furthermore, we need some key observations on the cancellations in the physical structures of the system \eqref{NS} for the flux terms and viscosity terms in order to close the a priori estimates, which is partially motivated by our recent time-asymptotic stability results in Li and Wang \cite{LW} and  Li, Wang and Wang \cite{LWW} for 2D/3D compressible viscous fluids, where it is proved that if the initial data is around the planar rarefaction wave data, then the 2D initial value problem \eqref{NS}-\eqref{NSI} has a unique global smooth solution that  goes to the planar rarefaction wave fan as $t \to \infty$ with the viscosity coefficients $\mu_1$ and $\lambda_1$ being fixed. In the present paper, our goal is  to justify the vanishing viscosity limit of 2D compressible Navier-Stokes equations \eqref{NS}-\eqref{NSI} to the planar rarefaction wave as the viscosity parameter $\varepsilon \to 0+$ and then both the viscosity coefficients $\mu_1, \lambda_1\to 0$. Compared with the time-asymptotic stability results of planar rarefaction wave in \cite{LW,LWW}, some new difficulties occur and the hyperbolic wave is crucially introduced to justify the vanishing viscosity limit.  More precisely, the detailed 2D vanishing viscosity limit result can be found in Theorem \ref{theorem1} below.
 
Next we describe the one-dimensional rarefaction wave to \eqref{ES} and the planar rarefaction wave to \eqref{2ES}.  The Euler system \eqref{ES} is   strictly hyperbolic   for $\rho >0$ with two distinct eigenvalues
\[
\lambda_1(\rho, u_1) = u_1 - \sqrt{p'(\rho)},   \qquad \lambda_2(\rho, u_1) = u_1 + \sqrt{p'(\rho)}.
\]
The two   right eigenvectors are denoted by $r_1(\rho, u_1)$ and $r_2(\rho, u_1)$,  
and the both characteristic fields are genuinely nonlinear, i.e.,  
$$\nabla_{(\rho, u_1)}\lambda_i(\rho, u_1)\cdot r_i(\rho, u_1)\neq 0$$
for any $\rho>0, u_1$ and $i=1,2.$
The $i$-Riemann invariant $z_i(\rho, u_1)~(i=1,2)$ to the Euler system \eqref{ES} is given by
\begin{equation}\label{ri}
z_i (\rho, u_1) = u_1 + (-1)^{i+1}\int^{\rho}\frac{\sqrt{p'(s)}}{s}ds,
\end{equation}
satisfying $\nabla_{(\rho, u_1)}z_i(\rho, u_1)\cdot r_i(\rho, u_1)\equiv0 ~(i=1,2)$ for all $\rho>0$ and $u_1$.
In this paper we  consider only the 2-rarefaction wave without loss of generality,  since the 1-rarefaction wave and the superposition of two rarefaction waves can be treated similarly. If the 2-Riemann invariant $z_2(\rho, u_1)$ is constant and the second eigenvalue $\lambda_2(\rho, u_1)$ is expanding along the 2-rarefaction wave curve,  i.e.,
\begin{equation} \label{RC}
u_{1+} - \int_{\rho_-}^{\rho_+}\frac{\sqrt{p'(s)}}{s}ds = u_{1-},  \qquad \lambda_2(\rho_+, u_{1+})>\lambda_2(\rho_-, u_{1-}),
\end{equation}
the Riemann problem \eqref{ES}-\eqref{ini} has a self-similar wave fan $(\rho^r, u_1^r)(\frac{x_1}{t})$   consisting of only the constant states and the centered 2-rarefaction waves (cf. \cite{L}). 
The planar rarefaction wave solution to the two-dimensional compressible Euler equations \eqref{2ES}-\eqref{2ini} is then defined as $(\rho^r, u_1^r,0)(\frac{x_1}{t})$.

Now we state our main result   as follows.

\begin{theorem} \label{theorem1}
	Let $(\rho^r, u_1^r, 0)(\frac{x_1}{t})$ be the planar 2-rarefaction wave to the 2D Euler system \eqref{2ES} which connects the constant states $(\rho_\pm, u_{1\pm}, 0)$ satisfying (\ref{RC}) with $\rho_\pm >0$ and $T>0$ be any arbitrarily large but fixed time. Then there exists a positive constant $\varepsilon_0$ such that for any $\varepsilon \in (0, \varepsilon_0)$, we can construct a family of smooth solutions $(\rho^\varepsilon, \textbf{u}^\varepsilon) = (\rho^\varepsilon, u_1^\varepsilon, u_2^\varepsilon)$ up to time $T$ with the initial value \eqref{PNSI} to the compressible Navier-Stokes equations \eqref{NS} satisfying
	\[
	\begin{cases} 
	(\rho^\varepsilon - \rho^r, u_1^\varepsilon - u_1^r, u_2^\varepsilon) \in C^0(0, T; L^2(\bbr \times \bbt)), \\
	(\nabla\rho^\varepsilon, \nabla\textbf{u}^\varepsilon) \in C^0(0, T; H^1(\bbr \times \bbt)), \\
	\nabla^3 \textbf{u}^\varepsilon \in L^2(0, T; L^2(\bbr \times \bbt)),
	\end{cases}
	\]
Moreover, for any small positive constant $h$, there exists a constant $C_{h,T}$ independent of $\varepsilon$, such that
	\begin{equation} \label{AB}
	\sup_{h \leq t \leq T} \big\|(\rho^\varepsilon, u^\varepsilon_1, u^\varepsilon_2)(t, x_1, x_2) - (\rho^r, u_1^r, 0)(\frac{x_1}{t})\big\|_{L^\infty(\bbr\times\bbt)} \leq C_{h,T} \varepsilon^{\frac 16}|\ln \varepsilon|.
	\end{equation}
 As the viscosities vanish, i.e. $\varepsilon \to 0$, the solution $(\rho^\varepsilon, \textbf{u}^\varepsilon) = (\rho^\varepsilon, u_1^\varepsilon, u_2^\varepsilon) (t, x_1, x_2)$ converges to the planar rarefaction wave fan $(\rho^r, u_1^r, 0)(\frac{x_1}{t})$ pointwisely except at the original point $(0, 0)$, and furthermore, $$(\rho^\varepsilon, \textbf{u}^\varepsilon)(t, x_1, x_2) \rightarrow (\rho^r, u_1^r, 0)(\frac{x_1}{t}), ~~~{\rm a.e.}~~ {\rm in}~~ \bbr^+ \times\bbr \times \bbt.$$
\end{theorem}

We remark that Theorem \ref{theorem1} gives the first vanishing viscosity result to the planar rarefaction wave with arbitrary strength for the multi-dimensional viscous system \eqref{NS} with physical viscosities, while the corresponding vanishing viscosity limit problems for the planar shock or contact discontinuity case are still completely open as far as we know.
%
To prove Theorem \ref{theorem1}, we first construct a smooth approximate rarefaction wave to the Euler system \eqref{ES} or \eqref{2ES} since the self-similar rarefaction wave fan is only Lipschitz continuous. The next crucial step is to introduce a new wave, called the hyperbolic wave, to recover the physical viscosities for the inviscid approximate rarefaction wave profile. Note that this hyperbolic wave plays an essential role for the vanishing viscosity limit of 2D compressible Navier-Stokes equations   towards the planar rarefaction wave and if we only use the inviscid 1D hyperbolic rarefaction wave profile as the ansatz without the hyperbolic wave  constructed, then $H^2$-norm of the perturbation of the solution to the 2D compressible Navier-Stokes equations  around the planar rarefaction wave is not uniform-in-$\varepsilon$ and consequently we can not justify the vanishing viscosity limit of planar rarefaction wave as in Theorem \ref{theorem1}.  Then the solution to the 2D compressible Navier-Stokes equations \eqref{NS} is sought around the superposition of both the rarefaction wave profile and the hyperbolic wave, and finally the vanishing viscosity limit to the planar rarefaction wave in \eqref{AB} is rigorously justified.
%
Note also that our vanishing viscosity analysis could also be applied to the vanishing viscosity limit to the superposition of 1-rarefaction wave and 2-rarefaction wave for the two-dimensional compressible Navier-Stokes equations \eqref{NS} provided we consider the wave interaction estimates additionally. 
We finally remark that the corresponding vanishing viscosity limit of the compressible Navier-Stokes equations \eqref{NS} to the planar rarefaction wave in the spatial three-dimensional case is still open and will be studied in our future investigation. 

The rest of the paper is organized as follows. In Section 2, we first construct the approximate rarefaction wave to the Euler system \eqref{ES} or \eqref{2ES} and then introduce the hyperbolic wave to recover the physical viscosities to the inviscid smooth approximate rarefaction wave. In Section 3, we reformulate the system as the perturbation of the solution to 2D compressible Navier-Stokes equations \eqref{NS} around the solution profile consisting of both the approximate rarefaction wave and the hyperbolic wave and then based on the a priori estimates, we prove our main Theorem \ref{theorem1}. Finally, in Section 4, we prove the a priori estimates for the perturbation system by using an elementary $L^2$ energy method.

Before concluding this introduction, we present some notations that will be used in this paper. We use $H^k(\bbr \times \bbt)$ and $H^k(\bbr \times \bbt_\varepsilon)(k \geq 0, k\in \mathbb{Z})$ to denote the usual Sobolev space with the norm $\|\cdot\|_k$, where $\bbt_\varepsilon:=\bbr/\frac1\varepsilon\bbz$ is the scaled torus. We denote $L^2(\bbr \times \bbt) = H^0(\bbr \times \bbt), L^2(\bbr \times \bbt_\varepsilon) = H^0(\bbr \times \bbt_\varepsilon)$ and   set $\|\cdot\| = \|\cdot\|_0$. For simplicity, we also write $C$ as generic positive constants which are independent of $\varepsilon, \delta$ and $T$, and $C_T$ as positive constants which are independent of $\varepsilon$ and $\delta$, but may depend on $T$.

\bigskip
%
%
\section{Construction of the Solution Profile}
\setcounter{equation}{0}

In this section we  construct the approximate rarefaction wave to the Euler system \eqref{ES} or \eqref{2ES} and   introduce the hyperbolic wave to recover the  physical viscosities to the inviscid smooth approximate rarefaction wave.

\subsection{Smooth Approximate Rarefaction Wave}

Since the rarefaction wave is only Lipschitz continuous, we will construct a smooth approximation rarefaction wave through the Burgers' equation as in \cite{X-1, H-L, H-L-W}.
Consider the Riemann problem for the inviscid Burgers' equation:
\begin{equation} \label{BE}
\begin{cases}
\displaystyle w_t + ww_{x_1} = 0, \\
\displaystyle w(0, x_1) = w_0^r(x_1) = \begin{cases}
w_-,  \quad x_1 < 0,\\
w_+,  \quad x_1 > 0.
\end{cases}
\end{cases}
\end{equation}
If $w_- < w_+$, then \eqref{BE} has the self-similar rarefaction wave fan $w^r(t, x_1) = w^r(x_1/t)$ given by
\begin{equation} \label{BES}
w^r(t, x_1) = w^r(\frac {x_1}{t}) = \begin{cases}
w_- , \qquad x_1 < w_-t, \\
\frac{x_1}{t}, \qquad w_-t \leq x_1 \leq w_+t, \\
w_+ , \qquad x_1 > w_+t.
\end{cases}
\end{equation}
As in \cite{H-L-W}, the approximate rarefaction wave to the   Navier-Stokes equations \eqref{NS} can be constructed using the smooth solution of the Burgers' equation:
\begin{equation} \label{ABE}
\begin{cases}
\displaystyle w_t + ww_{x_1} = 0, \\
\displaystyle w(0, x_1) = w_0(x_1) = \frac{w_+ + w_-}{2} + \frac{w_+ - w_-}{2} \tanh \frac{x_1}{\delta},
\end{cases}
\end{equation}
where $\delta>0$ is a small constant depending on the viscosity parameter $\varepsilon$. In fact, we take $\delta=\varepsilon^{\frac16}$ in the present paper.
The following properties can be proved by the characteristic method, see \cite{X-1, H-L-W}.
\begin{lemma} \label{lemma2.1}
  Suppose $w_+ > w_-$ and set $\tilde{w} = w_+ - w_-$. Then the problem \eqref{ABE} has a unique smooth global solution $w(t, x_1)$ such that
   
  (1)~$w_- < w(t, x_1) < w_+, ~w_{x_1} >0$ for $x_1 \in \bbr$ and $t \geq 0, \delta > 0$.
   
  (2)~The following estimates hold for all $t > 0, \delta > 0$ and $p \in [1, + \infty]$:
  \begin{align*}
    &\|w_{x_1}(t, \cdot)\|_{L^p(\mathbb{R})} \leq C \tilde{w}^{1/p} (\delta + t)^{-1+1/p}, \\
    &\|w_{x_1x_1}(t, \cdot)\|_{L^p(\mathbb{R})} \leq C (\delta + t)^{-1} \delta^{-1+1/p}, \\
    &\|w_{x_1x_1x_1}(t, \cdot)\|_{L^p(\mathbb{R})} \leq C (\delta + t)^{-1} \delta^{-2+1/p}, \\
    &|w_{x_1x_1}(t, x_1)| \leq \frac 4\delta w_{x_1}(t, x_1).
  \end{align*}
  
  (3)~There exists a constant $\delta_0 \in (0, 1)$ such that for $\delta \in (0, \delta_0] $ and $t > 0$,
  \[
   \|w(t, \cdot) - w^r(\frac{\cdot}{t})\|_{L^\infty(\mathbb{R})} \leq C \delta t^{-1} [\ln(1 + t) + |\ln \delta|].
  \] 
\end{lemma}

We now consider the approximate rarefaction wave for the Euler system \eqref{ES}-\eqref{ini}. From now on,  the constant states $(\rho_\pm, u_{1\pm})$ are fixed  and connected by the 2-rarefaction wave. Set $w_\pm = \lambda_2(\rho_\pm, u_{1\pm})$. 
In fact,  the 2-rarefaction wave $(\rho^r, u_1^r)(t, x_1) = (\rho^r, u_1^r)(x_1/t)$ to the Riemann problem \eqref{ES} - \eqref{RC} is given explicitly by
\begin{align*} 
  &\lambda_2(\rho^r, u_1^r)(t, x_1) = w^r(t, x_1), \\
  &z_2(\rho^r, u_1^r)(t, x_1) = z_2 (\rho_\pm, u_{1\pm}),
\end{align*}
where $z_2(\rho, u_1)$ is the 2-Riemann invariant defined in \eqref{ri}.
The corresponding  smooth approximate rarefaction wave $(\bar{\rho}, \bar{u}_1)(t, x_1)$ of the 2-rarefaction wave fan $(\rho^r, u_1^r)(\frac{x_1}t)$ can be constructed by
\begin{align}
  \begin{aligned} \label{AR}
    &\lambda_2(\bar{\rho}, \bar{u}_1)(t, x_1) = w(1+t, x_1), \\
    &z_2(\bar{\rho}, \bar{u}_1)(t, x_1) = z_2 (\rho_\pm, u_{1\pm}),
  \end{aligned}
\end{align}
where $w(t, x_1)$ is the smooth solution to the Burgers' equation in \eqref{ABE}.
It is easy to see that the above approximate rarefaction wave $(\bar{\rho}, \bar{u}_1)$ satisfies the following system:
\begin{equation}  \label{ANS}
  \begin{cases}
    \displaystyle \bar{\rho}_t + (\bar{\rho}\bar{u}_1)_{x_1} = 0, \\
    \displaystyle (\bar{\rho} \bar{u}_1)_t + (\bar{\rho} \bar{u}_1^2 + p(\bar{\rho}))_{x_1} = 0, \\
    (\bar{\rho}, \bar{u}_1)(0, x_1) := (\bar{\rho}_0, \bar{u}_{10})(x_1).
  \end{cases}
\end{equation}
The following lemma follows from Lemma \ref{lemma2.1} (cf. \cite{H-L-W}).
\begin{lemma} \label{lemma2.2}
	The smooth approximate 2-rarefaction wave $(\bar{\rho}, \bar{u}_1)$ defined in \eqref{AR} satisfies the following properties:
	
	(1)~$\bar{u}_{1x_1} = \frac{2}{\gamma + 1}w_{x_1} > 0$ for all $x_1 \in \bbr$ and $t \geq 0, \bar{\rho}_{x_1} = \bar{\rho}^{\frac{3 - \gamma}{2}} \bar{u}_{1x_1}>0$, and 
	$$\bar{\rho}_{x_1x_1} = \bar{\rho}^{\frac{3 - \gamma}{2}} \bar{u}_{1x_1x_1} + \frac{3-\gamma}{2}\bar{\rho}^{2-\gamma} (\bar{u}_{1x_1})^2.$$
	
	(2)~The following estimates hold for all $t \geq 0, \delta > 0$ and $p \in [1, + \infty]$:
	\begin{align*}
	&\|(\bar{\rho}_{x_1}, \bar{u}_{1x_1})\|_{L^p(\mathbb{R})} \leq C \tilde{w}^{1/p} (\delta + t)^{-1+1/p}, \\
	&\|(\bar{\rho}_{x_1x_1}, \bar{u}_{1x_1x_1})\|_{L^p(\mathbb{R})} \leq C (\delta + t)^{-1} \delta^{-1+1/p}, \\
	&\|(\bar{\rho}_{x_1x_1x_1}, \bar{u}_{1x_1x_1x_1})\|_{L^p(\mathbb{R})} \leq C (\delta + t)^{-1} \delta^{-2+1/p}.
	\end{align*}
	
	(3)~There exists a constant $\delta_0 \in (0, 1)$ such that for $\delta \in (0, \delta_0]$ and $t > 0$,
	\[
	\|(\bar{\rho}, \bar{u}_1)(t, \cdot) - (\rho^r, u_1^r)(\frac{\cdot}{t})\|_{L^\infty(\mathbb{R})} \leq C \delta t^{-1} [\ln(1 + t) + |\ln \delta|].
	\]
\end{lemma}

\subsection{Hyperbolic Wave} 

If we only choose the approximate rarefaction wave $(\bar{\rho}, \bar{u}_1)(t, x_1)$ as the approximate wave profile, the error terms arising from  the viscous terms in the approximate rarefaction wave are not good enough for obtaining the desired uniform estimates with respect to the viscosities. 
Thus we introduce the hyperbolic wave to recover the physical viscosities for the inviscid approximate rarefaction wave profile, which   a crucial   in our  analysis of vanishing viscosity limit and  partially motivated by  \cite{HWY-2}. 
We now provide a detailed description of this hyperbolic wave. Let the hyperbolic wave $(d_1, d_2)(t, x_1)$ satisfy the linear hyperbolic system
\begin{equation}  \label{HW}
\begin{cases}
\displaystyle d_{1t} + d_{2x_1} = 0, \\
\displaystyle d_{2t} + \left(-\frac{\bar{m}_1^2}{\bar{\rho}^2}d_1 + p'(\bar{\rho})d_1 + \frac{2\bar{m}_1}{\bar{\rho}}d_2\right)_{x_1} = (2\mu_1 + \lam_1)\bar{u}_{1x_1x_1} = (2\mu + \lam)\varepsilon\bar{u}_{1x_1x_1}, \\
(d_1, d_2)(0, x_1) = (0, 0),
\end{cases}
\end{equation}
where $\bar{m}_1 := \bar{\rho}\bar{u}_1$ represents the momentum of the approximate rarefaction wave. 
We shall solve this linear hyperbolic system \eqref{HW} on the fixed time interval $[0, T]$. 
We first  diagonalize the above system. Rewrite the system \eqref{HW} as
\[
\left[
\begin{array}{c}
\displaystyle d_1\\
\displaystyle d_2
\end{array}
\right]_t +
\left(
\bar{A}
\left[
\begin{array}{c}
\displaystyle d_1\\
\displaystyle d_2
\end{array}
\right]
\right)_{x_1} =
\left[
\begin{array}{c}
0\\
\displaystyle (2\mu + \lam)\varepsilon \bar{u}_{1x_1x_1}
\end{array}
\right],
\]
where
\[
\bar{A} =
\left[
\begin{array}{lc}
0 & 1\\
-\frac{\bar{m}_1^2}{\bar{\rho}^2} + p'(\bar{\rho}) & \frac{2\bar{m}_1}{\bar{\rho}}
\end{array}
\right]
\]
with two distinct eigenvalues $\bar{\lam}_1(\bar{\rho}, \bar{m}_1) = \frac{\bar{m}_1}{\bar{\rho}} - \sqrt{p'(\bar{\rho})}, \; \bar{\lam}_2(\bar{\rho}, \bar{m}_1) = \frac{\bar{m}_1}{\bar{\rho}} + \sqrt{p'(\bar{\rho})}$ and the corresponding left and right eigenvectors $\bar{l}_i, \bar{r}_i(i = 1, 2)$. 
For example, we can choose $\bar{l}_i(\bar{\rho}, \bar{m}_1) = \left(\frac{\sqrt{2}}{2}(\frac{-\bar{m}_1}{\bar{\rho}} + (-1)^i\sqrt{p'(\bar{\rho})}), \frac{\sqrt{2}}{2}\right), ~~
\bar{r}_i(\bar{\rho}, \bar{m}_1) = \left(\frac{(-1)^i\sqrt{2}}{2\sqrt{p'(\bar{\rho})}}, \frac{(-1)^i\sqrt{2}}{2\sqrt{p'(\bar{\rho})}}(\frac{\bar{m}_1}{\bar{\rho}} + (-1)^i\sqrt{p'(\bar{\rho})})\right)^\top$ satisfying
\[
\bar{L}\bar{A}\bar{R} = diag(\bar{\lam}_1, \bar{\lam}_2) := \bar{\Lambda}, \quad \bar{L}\bar{R} = I, 
\]
where $\bar{L} = (\bar{l}_1, \bar{l}_2)^\top, \bar{R} = (\bar{r}_1, \bar{r}_2)$ and $I$ is the $2 \times 2$ identity matrix. Now we set
\[
(D_1, D_2)^\top = \bar{L}(d_1, d_2)^\top,
\]
then
\[
(d_1, d_2)^\top = \bar{R}(D_1, D_2)^\top,
\]
and $(D_1, D_2)$ satisfies the diagolized system
\begin{equation}\label{DHW}
\left[
\begin{array}{c}
\displaystyle D_1\\
\displaystyle D_2
\end{array}
\right]_t +
\left(
\bar{\Lambda}
\left[
\begin{array}{c}
\displaystyle D_1\\
\displaystyle D_2
\end{array}
\right]
\right)_{x_1} = \bar{L}_t\bar{R}
\left[
\begin{array}{c}
\displaystyle D_1\\
\displaystyle D_2
\end{array}
\right] + \bar{L}_{x_1}\bar{A}\bar{R}
\left[
\begin{array}{c}
\displaystyle D_1\\
\displaystyle D_2
\end{array}
\right]+ \bar{L}
\left[
\begin{array}{c}
0\\
\displaystyle (2\mu + \lam)\varepsilon \bar{u}_{1x_1x_1}
\end{array}
\right].
\end{equation}
Since the 2-Riemann invariant is constant along the approximate 2-rarefaction wave curve, we have  
\begin{equation}\label{sc}
\bar{L}_t = -\bar{\lam}_2 \bar{L}_{x_1},
\end{equation}
which is a crucial structure to solve the linear hyperbolic system \eqref{DHW} in the interval $[0,T]$, otherwise, it does not seem obvious to solve easily the strongly coupled hyperbolic system \eqref{DHW} on the bounded domain  $[0,T]$.
Substituting the structure relation \eqref{sc} into \eqref{DHW}, we obtain the diagonalized system
\begin{equation}  \label{DHWS}
\begin{cases}
\displaystyle D_{1t} + (\bar{\lam}_1D_1)_{x_1} = \frac{\sqrt{2}}{2}(2\mu + \lam)\varepsilon \bar{u}_{1x_1x_1} + (a_{11}(\bar{\rho})\bar{\rho}_{x_1} + a_{12}(\bar{\rho})\bar{u}_{1x_1})D_1, \\[3mm]
\displaystyle D_{2t} + (\bar{\lam}_2D_2)_{x_1} = \frac{\sqrt{2}}{2}(2\mu + \lam)\varepsilon \bar{u}_{1x_1x_1} + (a_{21}(\bar{\rho})\bar{\rho}_{x_1} + a_{22}(\bar{\rho})\bar{u}_{1x_1})D_1, \\
(D_1, D_2)(0, x_1) = (0, 0).
\end{cases}
\end{equation}
In the diagonalized system \eqref{DHWS}, the equation of $D_1$ is decoupled with $D_2$ due to the rarefaction wave structure of the system as in \eqref{sc}. Therefore, we can solve $D_1$ first and then $D_2$ in \eqref{DHWS} by the standard characteristic method.  Furthermore, we have the following important estimates for the hyperbolic wave $(d_1, d_2)$:
\begin{lemma} \label{lemma2.3}
	There exists a positive constant $C_T$ independent of $\delta$ and $\varepsilon$, such that
	\[
	\|\frac{\partial^k}{\partial x_1^k}(d_1,d_2)(t, \cdot)\|_{L^2(\mathbb{R})}^2 \leq C_T (\frac{\varepsilon}{\delta^{k + 1}})^2,  \quad k=0,1,2,3.
	\]
	In particular, it holds that
	$$
	\sup_{t\in[0,T]}\|(d_1,d_2)(t, \cdot)\|_{L^\infty(\mathbb{R})}=O(\frac{\varepsilon}{\delta^{\frac 32}}).
	$$
	
\end{lemma}
\begin{proof}
	Multiplying the second equation of \eqref{DHWS} by $D_2$ and integrating the resulting equation over $[0, t]$ with $t \in (0, T)$ imply
	\begin{align*}
	&\int_{\bbr} \frac{D_2^2}{2}(t, x_1) dx_1 + \int_{0}^{t}\int_{\bbr} \bar{\lam}_{2x_1}\frac{D_2^2}{2} dx_1dt \\
	&= \int_{0}^{t}\int_{\bbr} \bigg[\frac{\sqrt{2}}{2}(2\mu + \lam)\varepsilon \bar{u}_{1x_1x_1}D_2 + (a_{21}(\bar{\rho})\bar{\rho}_{x_1} + a_{22}(\bar{\rho})\bar{u}_{1x_1})D_1D_2\bigg] dx_1dt \\
	&\leq C \int_{0}^{t}\int_{\bbr} D_2^2 dx_1dt + C \varepsilon^2 \int_{0}^{t}\int_{\bbr} \bar{u}_{1x_1x_1}^2 dx_1dt + \beta \int_{0}^{t}\int_{\bbr} \bar{u}_{1x_1}D_2^2 dx_1dt \\
	&\quad + C_\beta \int_{0}^{t}\int_{\bbr} \bar{u}_{1x_1}D_1^2 dx_1dt \\
	&\leq C \int_{0}^{t}\int_{\bbr} D_2^2 dx_1dt + C \varepsilon^2 \int_{0}^{t} \delta^{-1}(\delta + t)^{-2} dt + \beta \int_{0}^{t}\int_{\bbr} \bar{u}_{1x_1}D_2^2 dx_1dt \\
	&\quad + C_\beta \int_{0}^{t}\int_{\bbr} \bar{u}_{1x_1}D_1^2 dx_1dt \\
	&\leq C \int_{0}^{t}\int_{\bbr} D_2^2 dx_1dt + C (\frac{\varepsilon}{\delta})^2 + \beta \int_{0}^{t}\int_{\bbr} \bar{u}_{1x_1}D_2^2 dx_1dt + C_\beta \int_{0}^{t}\int_{\bbr} \bar{u}_{1x_1}D_1^2 dx_1dt.
	\end{align*}
	Choosing $\beta$ suitably small and using Gronwall's inequality give
	\begin{align}
	\begin{aligned} \label{1}
	\int_{\bbr} D_2^2 (t, x_1) dx_1 + \int_{0}^{t}\int_{\bbr} \bar{u}_{1x_1} D_2^2 dx_1dt
	\leq C_T (\frac{\varepsilon}{\delta})^2 + C_T \int_{0}^{t}\int_{\bbr} \bar{u}_{1x_1}D_1^2 dx_1dt.
	\end{aligned}
	\end{align}
	
	Now we   multiply the first equation of \eqref{DHWS} by $\bar{\rho}^ND_1$ with $N$   a sufficiently large positive constant to be determined, and integrate the resulting equation over $[0, t]$ with $t \in (0, T)$ to get
	\begin{align*}
	&\int_{\bbr} \bar{\rho}^N\frac{D_1^2}{2}(t, x_1) dx_1 + \int_{0}^{t}\int_{\bbr} N\bar{\rho}^N\bar{u}_{1x_1} D_1^2 dx_1dt \\
	&= \int_{0}^{t}\int_{\bbr} \bigg[\frac{\sqrt{2}}{2}(2\mu + \lam)\varepsilon \bar{\rho}^N\bar{u}_{1x_1x_1}D_1 + \bar{\rho}^N(a_{11}(\bar{\rho})\bar{\rho}_{x_1} + a_{12}(\bar{\rho})\bar{u}_{1x_1})D_1^2 - \bar{\rho}^N\lam_{1x_1}\frac{D_1^2}{2}\bigg] dx_1dt \\
	&\leq C \int_{0}^{t}\int_{\bbr} \bar{\rho}^ND_1^2 dx_1dt + C \varepsilon^2 \int_{0}^{t}\int_{\bbr} \bar{u}_{1x_1x_1}^2 dx_1dt + C \int_{0}^{t}\int_{\bbr} \bar{\rho}^N\bar{u}_{1x_1}D_1^2 dx_1dt \\
	&\leq C \int_{0}^{t}\int_{\bbr} \bar{\rho}^ND_1^2 dx_1dt + C (\frac{\varepsilon}{\delta})^2 + C \int_{0}^{t}\int_{\bbr} \bar{\rho}^N\bar{u}_{1x_1}D_1^2 dx_1dt.
	\end{align*}
	Choosing $N$ large enough and using Gronwall's inequality give
	\begin{align}
	\begin{aligned} \label{2}
	\int_{\bbr} D_1^2(t, x_1) dx_1 + \int_{0}^{t}\int_{\bbr} \bar{u}_{1x_1}D_1^2 dx_1dt
	\leq C_T (\frac{\varepsilon}{\delta})^2.
	\end{aligned}
	\end{align}
	Combining \eqref{1} and \eqref{2}, we can get
	\begin{align*}
	\int_{\bbr} (D_1^2 + D_2^2)(t, x_1) dx_1 + \int_{0}^{t}\int_{\bbr} \bar{u}_{1x_1}(D_1^2 + D_2^2) dx_1dt
	\leq C_T (\frac{\varepsilon}{\delta})^2.
	\end{align*}
Thus the case $k = 0$  in Lemma \ref{lemma2.3} is proved. The other cases $k = 1, 2, 3$ can be proved similarly by differentiating  the  system $k$ times with respect to $x_1$, and we omit the details. 
\end{proof}

\subsection{Approximate Solution Profile}

The approximate solution profile $(\tilde{\rho}, \tilde{u}_1)$ consisting of the rarefaction wave $(\bar\rho,\bar u_1)$ and the hyperbolic wave $(d_1,d_2)$ to the compressible Navier-Stokes equations can be defined by
\begin{equation}\label{AW0}
\tilde{\rho}(t, x_1) = (\bar{\rho} + d_1)(t, x_1), \quad \tilde{m}_1(t, x_1) = (\bar{m}_1 + d_2)(t, x_1) := \tilde{\rho}\tilde{u}_1(t, x_1).
\end{equation}
Then the approximate wave profile $(\tilde{\rho}, \tilde{u}_1)$ satisfies the system
\begin{equation}  \label{AW}
\begin{cases}
\displaystyle \tilde{\rho}_t + (\tilde{\rho}\tilde{u}_1)_{x_1} = 0, \\
\displaystyle (\tilde{\rho} \tilde{u}_1)_t + (\tilde{\rho} \tilde{u}_1^2 + p(\tilde{\rho}))_{x_1} = (2\mu + \lam)\varepsilon \bar{u}_{1x_1x_1} +  (\tilde{\rho} \tilde{u}_1^2 - \bar{\rho} \bar{u}_1^2 + \bar{u}_1^2d_1 - 2\bar{u}_1d_2)_{x_1} \\
\qquad \qquad \qquad \qquad \qquad \quad + (p(\tilde{\rho}) - p(\bar{\rho}) - p'(\bar{\rho})d_1)_{x_1},
\end{cases}
\end{equation}
with the initial data
\begin{equation} \label{AWI}
(\tilde{\rho}, \tilde{u}_1)(0, x_1) = (\bar{\rho}_0, \bar{u}_{10})(x_1).
\end{equation}

\bigskip

%
%
\section{Reformulation of the Problem}
\setcounter{equation}{0}

To prove Theorem \ref{theorem1}, the solution $(\rho^\varepsilon, u_1^\varepsilon, u_2^\varepsilon)$ to the system \eqref{NS} is constructed as the perturbation around the approximate wave profile $(\tilde{\rho}, \tilde{u}_1, 0)$ defined in \eqref{AW0} and \eqref{AW}. Set the perturbation around the approximate wave profile $(\tilde{\rho}, \tilde{u}_1, 0)(t, x_1)$ by
\begin{equation}\label{per}
\begin{array}{ll}
\di \phi(t, x_1, x_2) := \rho^\varepsilon(t, x_1, x_2) - \tilde{\rho}(t, x_1), \\
\di \Psi(t, x_1, x_2) = (\psi_1, \psi_2)^\top(t, x_1, x_2) := (u_1^\varepsilon, u_2^\varepsilon)^\top(t, x_1, x_2) - (\tilde{u}_1, 0)^\top(t, x_1),
\end{array}
\end{equation}
with $(\rho^\varepsilon, u_1^\varepsilon, u_2^\varepsilon)$ being the solution to the problem \eqref{NS} with the following initial data:
\begin{equation} \label{PNSI}
(\rho^\varepsilon, u_1^\varepsilon, u_2^\varepsilon)(0, x_1, x_2) := (\bar{\rho}_0, \bar{u}_{10}, 0)(x_1)+(\phi_0,\psi_{10}, \psi_{20})(x_1,x_2).
\end{equation}
For convenience, we reformulate the system by introducing a scaling for the independent variables. Set
\[
\tau = \frac{t}{\varepsilon}, \quad y_1 = \frac{x_1}{\varepsilon}, \quad y_2 = \frac{x_2}{\varepsilon}.
\]
For simplicity of notation, the superscription of $(\rho^\varepsilon, u_1^\varepsilon, u_2^\varepsilon)$ will be omitted as $(\rho, u_1, u_2)$ from now on if there is no confusion of notation. And here we still use the notations $(\rho, u_1, u_2)(\tau, y_1, y_2), (\tilde{\rho}, \tilde{u}_1)(\tau, y_1), (\bar{\rho}, \bar{u}_1)(\tau, y_1)$ and $(\phi,\Psi)(\tau, y_1, y_2)$ in the scaled independent variables, if without any confusion. 
From \eqref{NS} and \eqref{AW}, we obtain the following system for the perturbation $(\phi, \Psi):$
\begin{equation}  \label{REF}
    \begin{cases}
    \displaystyle \phi_{\tau} + \rho div\Psi + \rho_{y_2}\psi_2 + u_1\phi_{y_1} + \tilde{\rho}_{y_1}\psi_1 + \tilde{u}_{1y_1}\phi = 0, \\
    \displaystyle \rho \Psi_{\tau} + \rho u_1\Psi_{y_1} + \rho u_2 \Psi_{y_2} + (\rho \tilde{u}_{1y_1} \psi_1, 0)^\top + p'(\rho) \nabla\phi + ((p'(\rho) - \frac{\rho}{\tilde{\rho}} p'(\tilde{\rho})) \tilde{\rho}_{y_1}, 0)^\top \\
    \quad + ((2\mu + \lam)\frac{\bar{u}_{1y_1y_1}}{\tilde{\rho}}\phi, 0)^\top + (\frac{(\tilde{\rho} \tilde{u}_1^2 - \bar{\rho} \bar{u}_1^2 + \bar{u}_1^2d_1 - 2\bar{u}_1d_2)_{y_1}}{\tilde{\rho}}\phi, 0)^\top + (\frac{(p(\tilde{\rho}) - p(\bar{\rho}) - p'(\bar{\rho})d_1)_{y_1}}{\tilde{\rho}}\phi, 0)^\top \\
    = \mu\triangle\Psi + (\mu + \lam)\nabla div\Psi + ((2\mu + \lam)(\frac{-d_1\bar{u}_1 + d_2}{\tilde{\rho}})_{y_1y_1}, 0)^\top \\
    \quad - ((\tilde{\rho} \tilde{u}_1^2 - \bar{\rho} \bar{u}_1^2 + \bar{u}_1^2d_1 - 2\bar{u}_1d_2)_{y_1}, 0)^\top - ((p(\tilde{\rho}) - p(\bar{\rho}) - p'(\bar{\rho})d_1)_{y_1}, 0),
    \end{cases}
\end{equation}
\begin{equation} \label{REFI}
  (\phi, \Psi)(0, y_1, y_2) = (\phi, \psi_1, \psi_2)(0, y_1, y_2)= (\phi_0,\psi_{10}, \psi_{20})(y_1,y_2),
\end{equation}
where the initial perturbation is chosen to satisfy
\begin{equation}\label{ip}
\|(\phi_0,\psi_{10}, \psi_{20})(y_1,y_2)\|_{H^2(\mathbb{R}\times\mathbb{T}_\varepsilon)}=O(\varepsilon^{\frac16}).
\end{equation}
The solution of \eqref{REF}, \eqref{REFI} is sought in the set of functional space  $X(0, \frac{T}{\varepsilon})$, where for $0 \leq \tau_1 \leq \frac{T}{\varepsilon}$, we define
\begin{align*}
X(0, \tau_1) = \left\{ (\phi, \Psi)| \, (\phi, \Psi) \in C^0(0, \tau_1; H^2), \nabla \phi \in L^2(0, \tau_1; H^1), \nabla \Psi \in L^2(0, \tau_1; H^2)\right\}.
\end{align*}
We take $\delta = \varepsilon^a$ in what follows. By the estimate of the hyperbolic wave in Lemma \ref{lemma2.3}, we have
\[
|d_i| \leq C_T \frac{\varepsilon}{\delta^{3/2}} = C_T \varepsilon^{1 - \frac{3}{2}a} \leq \frac{1}{4}\rho_-, \qquad i=1,2,
\]
provided that $a < \frac{2}{3}$ and $\varepsilon \ll 1$. Then we have
\[
0 < \frac{3}{4}\rho_- = \rho_- - \frac{1}{4}\rho_- \leq \tilde{\rho} = \bar{\rho} + d_1 \leq \rho_+ + \frac{1}{4}\rho_-, \quad |\tilde{u}_1| \leq C,
\]
since $0 < \rho_- \leq \bar{\rho} \leq \rho_+, |\bar{u}_1| \leq C$.
In what follows, the analysis is always carried out under the a priori assumption
\begin{equation} \label{PA}
E = E(0, \tau_1(\varepsilon)) = \sup_{\tau \in [0, \tau_1(\varepsilon)]}\| (\phi, \Psi)(\tau) \|_2 \ll 1,
\end{equation}
where $[0, \tau_1(\varepsilon)]$ is the time interval in which the solution exists and it may depend on $\varepsilon$.
Under the a priori assumption \eqref{PA}, we can get
\begin{equation} \label{db}
	0 < \frac{1}{2} \rho_- = \frac{3}{4}\rho_- - \frac{1}{4}\rho_- \leq \rho = \phi + \tilde{\rho} \leq \rho_+ + \frac{1}{4}\rho_- + \frac{1}{4}\rho_- = \rho_+ + \frac{1}{2}\rho_-, \quad |\textbf{u}| \leq C,
\end{equation}
because we can take $E$ suitably small such that $|(\phi, \Psi)| \leq C \| (\phi, \Psi)(\tau) \|_2 \leq \frac{1}{4}\rho_-$.
The uniform   bounds  of the density  $\rho$ ensure that   the momentum equation $\eqref{NS}_2$ is strictly parabolic,  and thus    crucial for the local and global  existence of   classical solution of the system \eqref{NS}.

\begin{proposition} \label{proposition3.1}
There exists a positive constant $\varepsilon_0<1$ 
such that if $0 < \varepsilon \leq \varepsilon_0$, then the reformulated problem \eqref{REF}-\eqref{REFI} admits a unique solution $(\phi, \Psi) \in X(0, \frac{T}{\varepsilon})$ satisfying
  \begin{align}
  \begin{aligned} \label{PRO3.1}
    &\sup_{0 \leq \tau \leq  \frac{T}{\varepsilon}} \| (\phi, \Psi)(\tau) \|_2^2 + \int_{0}^{\frac{T}{\varepsilon}} \Big[\|\bar{u}_{1y_1}^{1/2} (\phi, \psi_1)\|^2 + \|(\nabla\phi, \nabla\Psi)\|_1^2 + \|\nabla^3\Psi\|^2 \Big]d\tau \\
    &\leq C_T \frac{\varepsilon}{\delta^4} + C \| (\phi_0, \Psi_0) \|_2^2,
  \end{aligned}
  \end{align}
where the constant $C_T$ is independent of $\varepsilon, \delta$, but may depend on $T$.
\end{proposition}

Once the Proposition \ref{proposition3.1} is proved, we have
$$
\begin{array}{ll}
  \di \sup_{0 \leq t \leq T} \| (\phi, \Psi)(t, x_1, x_2) \|_{L^\infty(\mathbb{R}\times\mathbb{T})} = \sup_{0 \leq \tau \leq  \frac{T}{\varepsilon}} \| (\phi, \Psi)(\tau, y_1, y_2) \|_{L^\infty(\mathbb{R}\times\mathbb{T}_\varepsilon)}\\
  \di  \leq C \sup_{0 \leq \tau \leq  \frac{T}{\varepsilon}} \| (\phi, \Psi)(\tau) \|_2 \leq C_T \frac{\varepsilon^{1/2}}{\delta^2} + C \| (\phi_0, \Psi_0) \|_2,
  \end{array}
$$
where we have used Sobolev imbedding $\|f\|_{L^\infty(\mathbb{R}\times\mathbb{T}_\varepsilon)}\leq C\|f\|_{H^2(\mathbb{R}\times\mathbb{T}_\varepsilon)}$ with the imbedding constant $C$ independent of $\varepsilon$ even though the domain $\mathbb{R}\times\mathbb{T}_\varepsilon$ depends on $\varepsilon$.
Thus we get
\begin{align*}
  &\|(\rho, u_1, u_2)(t, x_1, x_2) - (\rho^r, u_1^r, 0)(\frac{x_1}{t})\|_{L^\infty(\mathbb{R}\times\mathbb{T})} \\
  &\leq \| (\phi, \Psi)(t, x_1, x_2) \|_{L^\infty(\mathbb{R}\times\mathbb{T})} + C \| (d_1, d_2)(t, x_1) \|_{L^\infty(\mathbb{R})} + \|(\bar{\rho}, \bar{u}_1)(t, x_1) - (\rho^r, u_1^r)(\frac{x_1}{t})\|_{L^\infty(\mathbb{R})} \\
  &\leq C_T \frac{\varepsilon^{1/2}}{\delta^2} + C \varepsilon^{1/6} + C_T \frac{\varepsilon}{\delta^{3/2}} + C \delta t^{-1} [\ln(1 + t) + |\ln \delta|] \\
  &= C_T \varepsilon^{\frac{1}{2} - 2a} + C \varepsilon^{1/6} + C_T \varepsilon^{1 - \frac{3}{2}a} + C \varepsilon^a t^{-1} [\ln(1 + t) + |\ln \varepsilon|].
\end{align*}
Taking $a = \frac{1}{6}$, i.e. $\delta = \varepsilon^{1/6}$ and then the proof of Theorem \ref{theorem1} is completed.

The proof for the local existence and uniqueness of the classical solution to \eqref{REF}-\eqref{REFI} is standard (c.f.  \cite{So}),  especially  for the suitably small perturbation of the solution around the ansatz including both the planar rarefaction wave and the hyperbolic wave satisfying   \eqref{db},  and thus will be omitted. To prove Proposition \ref{proposition3.1}, it suffices to establish the following a priori estimates.
\begin{proposition}[a priori estimates] \label{proposition3.2}
Suppose that the reformulated problem \eqref{REF}-\eqref{REFI} has a solution $(\phi, \Psi) \in X(0, \tau_1(\varepsilon))$ for some $\tau_1(\varepsilon)(>0)$. Then there exists a positive constant $\varepsilon_1$ which is independent of $\varepsilon, \delta$ and $\tau_1(\varepsilon
)$, such that if $0 < \varepsilon \leq \varepsilon_1$ and $E(0, \tau_1(\varepsilon)) \ll 1$,
then it holds 
  \begin{align}
	\begin{aligned} \label{PRO3.2}
	  &\sup_{0 \leq \tau \leq  \tau_1(\varepsilon)} \| (\phi, \Psi)(\tau) \|_2^2 + \int_{0}^{\tau_1(\varepsilon)} \Big[\|\bar{u}_{1y_1}^{1/2} (\phi, \psi_1)\|^2 + \|(\nabla\phi, \nabla\Psi)\|_1^2 + \|\nabla^3\Psi\|^2 \Big]d\tau \\
	  &\leq C_T \frac{\varepsilon}{\delta^4} + C \| (\phi_0, \Psi_0) \|_2^2,
	\end{aligned}
  \end{align}
where the constant $C_T$ is independent of $\varepsilon$ and $\delta$, but may depend on $T$.
\end{proposition}

We note that   Theorem \ref{theorem1} follows once the Proposition \ref{proposition3.2} is proved. The remaining part of this paper, i.e.,  Section 4,  is devoted to the proof of Proposition \ref{proposition3.2}.

\bigskip

%
%
\section{A Priori Estimates}
\setcounter{equation}{0}

In this section, we shall prove Proposition \ref{proposition3.2}. Throughout this section 
we assume that  \eqref{RC} holds with fixed $\rho_\pm >0, u_{1\pm} \in \bbr$,   and   \eqref{REF}-\eqref{REFI} has a solution $(\phi, \Psi) \in X(0, \tau_1(\varepsilon))$ for some $\tau_1(\varepsilon)>0$. We use $C$ to denote a generic positive constant that may depend on $(\rho_\pm, u_{1\pm})$ but not  $\varepsilon, \delta$ and $T$,  and denote by $C_T$ as a generic positive constant  that may depend on $(\rho_\pm, u_{1\pm})$ and $T$ but not $\varepsilon$ and $\delta$. 
Set $E = \di\sup_{0 \leq \tau \leq \tau_1(\varepsilon)} \| (\phi, \Psi)(\tau) \|_2$.

\begin{lemma} \label{lemma4.1}
  There exists a positive constant $C_T$ such that for $0 \leq \tau \leq \tau_1(\varepsilon)$,
  \begin{align}
    \begin{aligned} \label{LEM4.1}
    \sup_{0 \leq \tau \leq \tau_1(\varepsilon)} \| (\phi, \Psi)(\tau) \|^2 + \int_{0}^{\tau_1(\varepsilon)}\big[ \|\bar{u}_{1y_1}^{1/2} (\phi,  \psi_1) \|^2 + \|\nabla \Psi\|^2\big] d\tau
    \leq C_T \frac{\varepsilon}{\delta^4} + C \| (\phi_0, \Psi_0) \|^2.
    \end{aligned}
  \end{align}
\end{lemma}
\begin{proof}
  First, multiplying the second equation of \eqref{REF} by $\Psi$ gives
  \begin{align}
    \begin{aligned} \label{01}
      &(\frac{1}{2}\rho|\Psi|^2)_\tau + \frac{1}{2}div(\rho\textbf{u}|\Psi|^2) - \mu div(\psi_i\nabla\psi_i) - (\mu + \lam) div(\Psi div\Psi) + \rho\tilde{u}_{1y_1}\psi_1^2 \\
      &\quad + \mu|\nabla\Psi|^2 + (\mu + \lam) (div\Psi)^2 + p'(\rho)\nabla\phi\cdot\Psi + (p'(\rho) - \frac{\rho}{\tilde{\rho}} p'(\tilde{\rho})) \tilde{\rho}_{y_1}\psi_1 \\
      &= (2\mu + \lam)(\frac{-d_1\bar{u}_1 + d_2}{\tilde{\rho}})_{y_1y_1}\psi_1 - (\tilde{\rho} \tilde{u}_1^2 - \bar{\rho} \bar{u}_1^2 + \bar{u}_1^2d_1 - 2\bar{u}_1d_2)_{y_1}\psi_1 \\
      &\quad - (p(\tilde{\rho}) - p(\bar{\rho}) - p'(\bar{\rho})d_1)_{y_1}\psi_1 - (2\mu + \lam)\frac{\bar{u}_{1y_1y_1}}{\tilde{\rho}}\phi\psi_1 \\
      &\quad - \frac{(\tilde{\rho} \tilde{u}_1^2 - \bar{\rho} \bar{u}_1^2 + \bar{u}_1^2d_1 - 2\bar{u}_1d_2)_{y_1}}{\tilde{\rho}}\phi\psi_1 - \frac{(p(\tilde{\rho}) - p(\bar{\rho}) - p'(\bar{\rho})d_1)_{y_1}}{\tilde{\rho}}\phi\psi_1.
    \end{aligned}
  \end{align}
Define the potential energy by
  \[
    \Phi(\rho , \tilde{\rho}) := \int_{\tilde{\rho}}^{\rho} \frac{p(s) - p(\tilde{\rho})}{s^2} ds
 = \frac{1}{(\gamma-1)\rho}(p(\rho) - p(\tilde{\rho}) - p'(\tilde{\rho})\phi).
\]
Direct computations yield  
  \begin{align}
    \begin{aligned} \label{02}
      &(\rho\Phi)_\tau + div(\rho\textbf{u}\Phi + (p(\rho) - p(\tilde{\rho}))\Psi) + \tilde{u}_{1y_1}(p(\rho) - p(\tilde{\rho}) - p'(\tilde{\rho})\phi) - p'(\rho)\nabla\phi\cdot\Psi \\
      &\quad - (p'(\rho) - \frac{\rho}{\tilde{\rho}} p'(\tilde{\rho})) \tilde{\rho}_{y_1}\psi_1 = 0.
    \end{aligned}
  \end{align}
Combining \eqref{01} and \eqref{02} together and then integrating the resulting equation over $[0, ~\tau]\times\bbr\times\bbt_\varepsilon$ imply
  \begin{align}
	\begin{aligned} \label{03}
	  &\|(\phi, \Psi)\|^2(\tau) + \int_{0}^{\tau} \big[\|\bar{u}_{1y_1}^{1/2} (\phi, \psi_1) \|^2 + \|\nabla \Psi\|^2\big] d\tau \\
	  &\leq C \| (\phi_0, \Psi_0) \|^2 + C \Big|\int_{0}^{\tau}\int_{\bbt_\varepsilon}\int_{\bbr} (2\mu + \lam)(\frac{-d_1\bar{u}_1 + d_2}{\tilde{\rho}})_{y_1}\psi_{1y_1} \\
	  &\quad - (\tilde{\rho} \tilde{u}_1^2 - \bar{\rho} \bar{u}_1^2 + \bar{u}_1^2d_1 - 2\bar{u}_1d_2)_{y_1}\psi_1 - (p(\tilde{\rho}) - p(\bar{\rho}) - p'(\bar{\rho})d_1)_{y_1}\psi_1 - (2\mu + \lam)\frac{\bar{u}_{1y_1y_1}}{\tilde{\rho}}\phi\psi_1 \\
	  &\quad - \frac{(\tilde{\rho} \tilde{u}_1^2 - \bar{\rho} \bar{u}_1^2 + \bar{u}_1^2d_1 - 2\bar{u}_1d_2)_{y_1}}{\tilde{\rho}}\phi\psi_1 - \frac{(p(\tilde{\rho}) - p(\bar{\rho}) - p'(\bar{\rho})d_1)_{y_1}}{\tilde{\rho}}\phi\psi_1 \\
	  &\quad - (\frac{-d_1\bar{u}_1 + d_2}{\tilde{\rho}})_{y_1}(p(\rho) - p(\tilde{\rho}) - p'(\tilde{\rho})\phi) - \rho(\frac{-d_1\bar{u}_1 + d_2}{\tilde{\rho}})_{y_1}\psi_1^2 dy_1dy_2d\tau \Big|,
	\end{aligned}
  \end{align}
where we have used the integration by parts
\[
\begin{array}{ll}
\di \int_{0}^{\tau}\int_{\bbt_\varepsilon}\int_{\bbr} (2\mu + \lam)(\frac{-d_1\bar{u}_1 + d_2}{\tilde{\rho}})_{y_1y_1}\psi_1 dy_1dy_2d\tau
\\ [4mm]
\di  = -\int_{0}^{\tau}\int_{\bbt_\varepsilon}\int_{\bbr} (2\mu + \lam)(\frac{-d_1\bar{u}_1 + d_2}{\tilde{\rho}})_{y_1}\psi_{1y_1} dy_1dy_2d\tau.
 \end{array}
\]
By Young's inequality, Lemma \ref{lemma2.2} and Lemma \ref{lemma2.3}, one has
  \begin{align*}
    &C \Big|\int_{0}^{\tau}\int_{\bbt_\varepsilon}\int_{\bbr} (2\mu + \lam)(\frac{-d_1\bar{u}_1 + d_2}{\tilde{\rho}})_{y_1}\psi_{1y_1} dy_1dy_2d\tau\Big| \\
    &\leq \frac{1}{16}\int_{0}^{\tau} \|\psi_{1y_1}\|^2 d\tau + C \int_{0}^{\tau}\int_{\bbt_\varepsilon}\int_{\bbr} |(\frac{-d_1\bar{u}_1 + d_2}{\tilde{\rho}})_{y_1}|^2 dy_1dy_2d\tau \\
    &\leq \frac{1}{16}\int_{0}^{\tau} \|\psi_{1y_1}\|^2 d\tau + C \varepsilon^{-1} \int_{0}^{t}\int_{\bbr} |(\frac{-d_1\bar{u}_1 + d_2}{\tilde{\rho}})_{x_1}|^2 dx_1dt \\
    &\leq \frac{1}{16}\int_{0}^{\tau} \|\psi_{1y_1}\|^2 d\tau + C_T \frac{\varepsilon}{\delta^4} + C_T \frac{\varepsilon^3}{\delta^7},
  \end{align*}
where we have used the following facts
  \[
  \int_{0}^{t}\int_{\bbr} |(\frac{d_{1x_1}\bar{u}_1}{\tilde{\rho}})|^2 dx_1dt \leq C \int_{0}^{t}\int_{\bbr} |d_{1x_1}|^2 dx_1dt \leq C_T \frac{\varepsilon^2}{\delta^4},
  \]
  \[
  \int_{0}^{t}\int_{\bbr} |(\frac{d_1\bar{u}_{1x_1}}{\tilde{\rho}})|^2 dx_1dt \leq C \int_{0}^{t} \|d_1\|_{L_{x_1}^2}^2 \|\bar{u}_{1x_1}\|_{L_{x_1}^{\infty}}^2 dt \leq C_T \frac{\varepsilon^2}{\delta^4},
  \]
  and
  \[
  \int_{0}^{t}\int_{\bbr} |(\frac{d_1\bar{u}_1d_{1x_1}}{\tilde{\rho}^2})|^2 dx_1dt \leq C \int_{0}^{t} \|d_1\|_{L_{x_1}^{\infty}}^2 \|d_{1x_1}\|_{L_{x_1}^2}^2 dt \leq C_T \frac{\varepsilon^4}{\delta^7}.
  \]
By the one-dimensional Sobolev's inequality, Young's inequality, Lemma \ref{lemma2.2} and Lemma \ref{lemma2.3}, it holds that
\begin{align*}
&C \Big|\int_{0}^{\tau}\int_{\bbt_\varepsilon}\int_{\bbr} (\tilde{\rho} \tilde{u}_1^2 - \bar{\rho} \bar{u}_1^2 + \bar{u}_1^2d_1 - 2\bar{u}_1d_2)_{y_1}\psi_1 dy_1dy_2d\tau\Big| \\
&= C \Big|\int_{0}^{\tau}\int_{\bbt_\varepsilon}\int_{\bbr} (\frac{(d_1\bar{u}_1 - d_2)^2}{\tilde{\rho}})_{y_1}\psi_1 dy_1dy_2d\tau\Big| \\
&\leq C \int_{0}^{\tau}\int_{\bbt_\varepsilon} \|(\frac{(d_1\bar{u}_1 - d_2)^2}{\tilde{\rho}})_{y_1}\|_{L_{y_1}^1}\|\psi_1\|_{L_{y_1}^\infty} dy_2d\tau \\
&\leq C \int_{0}^{\tau}\int_{\bbt_\varepsilon}\|(\frac{(d_1\bar{u}_1 - d_2)^2}{\tilde{\rho}})_{y_1}\|_{L_{y_1}^1} \|\psi_1\|_{L_{y_1}^2}^{1/2}\|\psi_{1y_1}\|_{L_{y_1}^2}^{1/2} dy_2d\tau \\
&\leq  \frac{1}{16}\int_{0}^{\tau} \|\psi_{1y_1}\|^2 d\tau + C \int_{0}^{\tau} \|(\frac{(d_1\bar{u}_1 - d_2)^2}{\tilde{\rho}})_{y_1}\|_{L_{y_1}^1}^{4/3} \int_{\bbt_\varepsilon} \|\psi_1\|_{L_{y_1}^2}^{2/3} dy_2d\tau \\
&\leq  \frac{1}{16}\int_{0}^{\tau} \|\psi_{1y_1}\|^2 d\tau + C \varepsilon^{-2/3}\int_{0}^{\tau} \|(\frac{(d_1\bar{u}_1 - d_2)^2}{\tilde{\rho}})_{y_1}\|_{L_{y_1}^1}^{4/3} \|\psi_1\|^{2/3} d\tau \\
&\leq  \frac{1}{16}\int_{0}^{\tau} \|\psi_{1y_1}\|^2 d\tau + C_T \varepsilon^{-5/3} \sup_{0 \leq t \leq T}\|(\frac{(d_1\bar{u}_1 - d_2)^2}{\tilde{\rho}})_{x_1}\|_{L_{x_1}^1}^{4/3} \sup_{0 \leq \tau \leq \tau_1(\varepsilon)}\|\psi_1\|^{2/3} \\
&\leq  \frac{1}{16}\int_{0}^{\tau} \|\psi_{1y_1}\|^2 d\tau + \frac{1}{16}\sup_{0 \leq \tau \leq \tau_1(\varepsilon)}\|\psi_1\|^2 + C_T \varepsilon^{-5/2} \sup_{0 \leq t \leq T}\|(\frac{(d_1\bar{u}_1 - d_2)^2}{\tilde{\rho}})_{x_1}\|_{L_{x_1}^1}^2 \\
&\leq  \frac{1}{16}\int_{0}^{\tau} \|\psi_{1y_1}\|^2 d\tau + \frac{1}{16}\sup_{0 \leq \tau \leq \tau_1(\varepsilon)}\|\psi_1\|^2 + C_T \frac{\varepsilon^{3/2}}{\delta^6},
\end{align*}
where we have used the following facts
\[
\|\frac{\bar{u}_1^2}{\tilde{\rho}^2}d_1^2\bar{\rho}_{x_1}\|_{L_{x_1}^1}\leq C \|\bar{\rho}_{x_1}\|_{L_{x_1}^\infty}\|d_1\|_{L_{x_1}^2}^2 \leq C_T \frac{\varepsilon^2}{\delta^3},
\]
\[
\|\frac{\bar{u}_1^2}{\tilde{\rho}^2}d_1^2d_{1x_1}\|_{L_{x_1}^1} \leq C \|d_1\|_{L_{x_1}^\infty}\|d_1\|_{L_{x_1}^2}\|d_{1x_1}\|_{L_{x_1}^2} \leq C_T \frac{\varepsilon^3}{\delta^{9/2}},
\]
and
\[
\|\frac{2\bar{u}_1^2}{\tilde{\rho}}d_1d_{1x_1}\|_{L_{x_1}^1} \leq C \|d_1\|_{L_{x_1}^2}\|d_{1x_1}\|_{L_{x_1}^2} \leq C_T \frac{\varepsilon^2}{\delta^3}.
\]
Using  the   Sobolev  inequality, H\"{o}lder  inequality, Young  inequality, and Lemma \ref{lemma2.2}, one has
\begin{align*}
&C \Big|\int_{0}^{\tau}\int_{\bbt_\varepsilon}\int_{\bbr} (2\mu + \lam)\frac{\bar{u}_{1y_1y_1}}{\tilde{\rho}}\phi\psi_1 dy_1dy_2d\tau\Big| \\
&\leq C \int_{0}^{\tau}\int_{\bbt_\varepsilon} \|\bar{u}_{1y_1y_1}\phi\|_{L_{y_1}^1} \|\psi_1\|_{L_{y_1}^\infty} dy_2d\tau \\
&\leq C \int_{0}^{\tau}\int_{\bbt_\varepsilon} \|\bar{u}_{1y_1y_1}\|_{L_{y_1}^2} \|\phi\|_{L_{y_1}^2} \|\psi_1\|_{L_{y_1}^2}^{1/2}\|\psi_{1y_1}\|_{L_{y_1}^2}^{1/2} dy_2d\tau \\
&\leq  \frac{1}{16}\int_{0}^{\tau} \|\psi_{1y_1}\|^2 d\tau + C \int_{0}^{\tau} \|\bar{u}_{1y_1y_1}\|_{L_{y_1}^2}^{4/3} \int_{\bbt_\varepsilon}  \|\phi\|_{L_{y_1}^2}^{4/3} \|\psi_1\|_{L_{y_1}^2}^{2/3} dy_2d\tau \\
&\leq  \frac{1}{16}\int_{0}^{\tau} \|\psi_{1y_1}\|^2 d\tau + C \int_{0}^{\tau} \|\bar{u}_{1y_1y_1}\|_{L_{y_1}^2}^{4/3} (\|\phi\|^2 + \|\psi_1\|^2) d\tau \\
&\leq  \frac{1}{16}\int_{0}^{\tau} \|\psi_{1y_1}\|^2 d\tau + C \sup_{0 \leq \tau \leq \tau_1(\varepsilon)}(\|\phi\|^2 + \|\psi_1\|^2) \int_{0}^{\tau} \|\bar{u}_{1y_1y_1}\|_{L_{y_1}^2}^{4/3} d\tau \\
&\leq  \frac{1}{16}\int_{0}^{\tau} \|\psi_{1y_1}\|^2 d\tau + C \varepsilon \sup_{0 \leq \tau \leq \tau_1(\varepsilon)}(\|\phi\|^2 + \|\psi_1\|^2) \int_{0}^{t} \|\bar{u}_{1x_1x_1}\|_{L_{x_1}^2}^{4/3} dt \\
&\leq  \frac{1}{16}\int_{0}^{\tau} \|\psi_{1y_1}\|^2 d\tau + C \frac{\varepsilon}{\delta} \sup_{0 \leq \tau \leq \tau_1(\varepsilon)}(\|\phi\|^2 + \|\psi_1\|^2).
\end{align*}
It follows from Lemma \ref{lemma2.2} and Lemma \ref{lemma2.3} that
\begin{align*}
&C \Big|\int_{0}^{\tau}\int_{\bbt_\varepsilon}\int_{\bbr} (\frac{-d_1\bar{u}_1 + d_2}{\tilde{\rho}})_{y_1}(p(\rho) - p(\tilde{\rho}) - p'(\tilde{\rho})\phi) dy_1dy_2d\tau\Big| \\
&\leq C \sup_{0 \leq \tau \leq \tau_1(\varepsilon)}\|(\frac{-d_1\bar{u}_1 + d_2}{\tilde{\rho}})_{y_1}\|_{L_{y_1}^\infty} \int_{0}^{\tau}\int_{\bbt_\varepsilon}\int_{\bbr} \phi^2 dy_1dy_2d\tau \\
&\leq C_T \sup_{0 \leq t \leq T}\|(\frac{-d_1\bar{u}_1 + d_2}{\tilde{\rho}})_{x_1}\|_{L_{x_1}^\infty} \sup_{0 \leq \tau \leq \tau_1(\varepsilon)} \|\phi\|^2 \\
&\leq C_T (\frac{\varepsilon}{\delta^{5/2}} + \frac{\varepsilon^2}{\delta^4}) \sup_{0 \leq \tau \leq \tau_1(\varepsilon)} \|\phi\|^2.
\end{align*}
The other terms in \eqref{03} can be estimated similarly and the details will be omitted for brevity.
Substituting these estimates into \eqref{03} and taking $\frac{\varepsilon}{\delta^4}$ and $\varepsilon$ suitably small, we can prove \eqref{LEM4.1} in Lemma \ref{lemma4.1}.
\begin{remark}
It should be remarked that the hyperbolic wave $(d_1,d_2)$ are crucially used in Lemma \ref{lemma4.1}, otherwise, the estimate \eqref{LEM4.1} in Lemma \ref{lemma4.1} would not be uniform in $\varepsilon$  if we just use the approximate rarefaction wave $(\bar\rho, \bar u)$ as the ansatz, which is quite different from the vanishing viscosity limit to the rarefaction wave for compressible Navier-Stokes equations in the one-dimensional case where the hyperbolic wave is not needed to justify the limit process. 
\end{remark}
\end{proof}

\begin{lemma} \label{lemma4.2}
  There exists a positive constant $C_T$ such that for $0 \leq \tau \leq \tau_1(\varepsilon)$,
  \begin{align}
	\begin{aligned} \label{LEM4.2}
 	  &\sup_{0 \leq \tau \leq \tau_1(\varepsilon)} (\| (\phi, \Psi)(\tau) \|^2 + \|\nabla\phi(\tau)\|^2) + \int_{0}^{\tau_1(\varepsilon)}\big[ \|\bar{u}_{1y_1}^{1/2} (\phi,  \psi_1) \|^2 + \|(\nabla\phi, \nabla\Psi)\|^2 \big] d\tau \\
 	  &\leq C_T \frac{\varepsilon}{\delta^4} + C (\| (\phi_0, \Psi_0) \|^2 + \|\nabla\phi_0\|^2) + C E^2 \int_{0}^{\tau_1(\varepsilon)} \|\nabla^2\Psi\|^2 d\tau.
	\end{aligned}
  \end{align}
\end{lemma}
\begin{proof}
  Applying the operator $\nabla$ to the first equation of \eqref{REF} and then multiplying the resulting equation by $\frac{\nabla\phi}{\rho^2}$ yield
 \begin{align}
 \begin{aligned} \label{11}
 &(\frac{|\nabla\phi|^2}{2\rho^2})_\tau + div(\frac{\textbf{u} |\nabla\phi|^2}{2\rho^2}) + \frac{\nabla{\rm div} \Psi\cdot \nabla\phi}{\rho}\\
 &= - \frac{\phi_{y_i} \nabla\phi\cdot\nabla\psi_i}{\rho^2} + \frac{|\nabla\phi|^2 div\Psi}{2\rho^2} - \frac{\tilde{\rho}_{y_1}\phi_{y_1} div\Psi}{\rho^2} - \frac{\tilde{\rho}_{y_1} \nabla\phi\cdot\nabla\psi_1}{\rho^2} - \frac{\tilde{u}_{1y_1} \phi_{y_1}^2}{\rho^2} \\
 &\quad + \frac{\tilde{u}_{1y_1}|\nabla\phi|^2}{2\rho^2} - \frac{\tilde{\rho}_{y_1y_1}\psi_1\phi_{y_1}}{\rho^2} - \frac{\tilde{u}_{1y_1y_1}\phi \phi_{y_1}}{\rho^2} \\
 &:= G(\tau, y_1, y_2).
 \end{aligned}
 \end{align}
 Multiplying the second equation of \eqref{REF} by $\frac{\nabla\phi}{\rho}$ gives
  \begin{align}
    \begin{aligned} \label{12}
      &(\Psi \cdot \nabla\phi)_\tau - div(\Psi\phi_\tau) + div(u_i\phi\Psi_{y_i}) - div(\textbf{u}\phi div\Psi) - div(\frac{\mu}{\rho}\nabla\psi_i\phi_{y_i}) \\
      &\quad + (\frac{\mu}{\rho}\nabla\psi_i\cdot\nabla\phi)_{y_i} + \frac{p'(\rho)}{\rho}|\nabla\phi|^2 - \frac{(2\mu + \lam)\nabla div\Psi\cdot\nabla\phi}{\rho} \\
      &= \tilde{\rho}(div\Psi)^2 + \phi\nabla\psi_i\cdot\Psi_{y_i} + \tilde{\rho}_{y_1}\psi_1div\Psi + \tilde{u}_{1y_1}\phi\psi_{1y_1} - \tilde{u}_{1y_1}\psi_1\phi_{y_1}  \\
      &\quad - (\frac{p'(\rho)}{\rho} - \frac{p'(\tilde{\rho})}{\tilde{\rho}})\tilde{\rho}_{y_1}\phi_{y_1}+ \frac{\mu}{\rho^2}\tilde{\rho}_{y_1}\Psi_{y_1}\cdot\nabla\phi - \frac{\mu}{\rho^2}\tilde{\rho}_{y_1}\nabla\psi_1\cdot\nabla\phi   \\
      &\quad + \frac{2\mu + \lam}{\rho}(\frac{-d_1\bar{u}_1 + d_2}{\tilde{\rho}})_{y_1y_1}\phi_{y_1}- \frac{(\tilde{\rho} \tilde{u}_1^2 - \bar{\rho} \bar{u}_1^2 + \bar{u}_1^2d_1 - 2\bar{u}_1d_2)_{y_1}}{\rho}\phi_{y_1}\\
      &\quad\di - \frac{(p(\tilde{\rho}) - p(\bar{\rho}) - p'(\bar{\rho})d_1)_{y_1}}{\rho}\phi_{y_1} - \frac{2\mu + \lam}{\tilde{\rho}\rho}\bar{u}_{1y_1y_1}\phi\phi_{y_1} \\
      &\quad - \frac{(\tilde{\rho} \tilde{u}_1^2 - \bar{\rho} \bar{u}_1^2 + \bar{u}_1^2d_1 - 2\bar{u}_1d_2)_{y_1}}{\tilde{\rho}\rho}\phi\phi_{y_1} - \frac{(p(\tilde{\rho}) - p(\bar{\rho}) - p'(\bar{\rho})d_1)_{y_1}}{\tilde{\rho}\rho}\phi\phi_{y_1} \\
      &:= \tilde{\rho}(div\Psi)^2 + H(\tau, y_1, y_2),
    \end{aligned}
  \end{align}
here we have used the fact
\begin{equation*}
\begin{array}{ll}
\di \frac{\mu\triangle\Psi\cdot\nabla\phi}{\rho} = div(\frac{\mu}{\rho}\nabla\psi_i\phi_{y_i}) - (\frac{\mu}{\rho}\nabla\psi_i\cdot\nabla\phi)_{y_i} \\[3mm]
\di\qquad\qquad\qquad + \frac{\mu\nabla div\Psi\cdot\nabla\phi}{\rho} + \frac{\mu}{\rho^2}\tilde{\rho}_{y_1}\Psi_{y_1}\cdot\nabla\phi - \frac{\mu}{\rho^2}\tilde{\rho}_{y_1}\nabla\psi_1\cdot\nabla\phi.
\end{array}
\end{equation*}
We note that some cancellations will occur to the last terms on the left hand side of both \eqref{11} and \eqref{12}  when we multiply   \eqref{11} by $2\mu + \lam$ and then add them together as in \cite{LW}.
We shall  use  the cancellations to close the a priori estimates.
 
Thus we multiply   \eqref{11} by $2\mu + \lam$,   add  the resulting equation and \eqref{12} together,   then integrate the final equation over $[0, ~\tau]\times\bbr\times\bbt_\varepsilon$ to obtain
  \begin{align}
    \begin{aligned} \label{13}
      &\int_{\bbt_\varepsilon}\int_{\bbr} \big(\frac{2\mu + \lam}{2\rho^2}|\nabla\phi|^2 + \Psi\cdot\nabla\phi \big) dy_1dy_2\big|_0^\tau + \int_{0}^{\tau}\int_{\bbt_\varepsilon}\int_{\bbr} \frac{p'(\rho)}{\rho}|\nabla\phi|^2 dy_1dy_2d\tau \\
      &= \int_{0}^{\tau}\int_{\bbt_\varepsilon}\int_{\bbr} \Big[(2\mu + \lam)G(\tau, y_1, y_2) + \tilde{\rho}(div\Psi)^2 + H(\tau, y_1, y_2) \Big] dy_1dy_2d\tau.
    \end{aligned}
  \end{align}
  Combining  \eqref{LEM4.1} and \eqref{13} leads to  
  \begin{align}
    \begin{aligned} \label{14}
      &\| (\phi, \Psi)(\tau) \|^2 + \|\nabla\phi(\tau)\|^2 + \int_{0}^{\tau}\big[ \|\bar{u}_{1y_1}^{1/2} (\phi,  \psi_1) \|^2 + \|(\nabla\phi, \nabla\Psi)\|^2 \big] d\tau \\
      &\leq C_T \frac{\varepsilon}{\delta^4} + C (\| (\phi_0, \Psi_0) \|^2 + \|\nabla\phi_0\|^2) \\
      &\quad + C\big|\int_{0}^{\tau}\int_{\bbt_\varepsilon}\int_{\bbr} \Big[(2\mu + \lam)G(\tau, y_1, y_2) + H(\tau, y_1, y_2) \Big] dy_1dy_2d\tau\big|,
    \end{aligned}
  \end{align}
  where $G$ and $H$ are defined in \eqref{11} and \eqref{12}, respectively.
  Here we just estimate some typical terms on the right-hand side of \eqref{14} in $G$ and $H$ for simplicity. First, by H\"{o}lder's inequality, Sobolev's inequality and Young's inequality, it holds that
  \begin{align*}
    &C\Big|\int_{0}^{\tau}\int_{\bbt_\varepsilon}\int_{\bbr} \frac{2\mu + \lam}{\rho^2}\phi_{y_i} \nabla\phi \cdot \nabla\psi_i dy_1dy_2d\tau\Big| \\
    &\leq C \int_{0}^{\tau} \|\nabla\phi\|\|\nabla\phi\|_{L^4}\|\nabla\Psi\|_{L^4} d\tau \leq C \int_{0}^{\tau} \|\nabla\phi\|\|\nabla\phi\|_1\|\nabla\Psi\|_1 d\tau \\
    & \leq C E \int_{0}^{\tau} \|\nabla\phi\|\|\nabla\Psi\|_1 d\tau \leq \frac{1}{160} \int_{0}^{\tau} \|\nabla\phi\|^2 d\tau + CE^2 \int_{0}^{\tau} \|\nabla\Psi\|_1^2 d\tau,
  \end{align*}
where in the second inequality we have used Sobolev imbedding $\|f\|_{L^4(\mathbb{R}\times\mathbb{T}_\varepsilon)}\leq C\|f\|_{H^1(\mathbb{R}\times\mathbb{T}_\varepsilon)}$ with the imbedding constant $C$ independent of $\varepsilon$.
Then it follows from Young's inequality, Lemma \ref{lemma2.2} and Lemma \ref{lemma2.3} that
  \begin{align*}
  &C\Big|\int_{0}^{\tau}\int_{\bbt_\varepsilon}\int_{\bbr} \frac{2\mu + \lam}{\rho^2} \tilde{\rho}_{y_1}\phi_{y_1} div\Psi dy_1dy_2d\tau\Big|
  \leq \frac{1}{160} \int_{0}^{\tau} \|\nabla\Psi\|^2 d\tau + C \int_{0}^{\tau} \|\tilde{\rho}_{y_1}\phi_{y_1}\|^2 d\tau \\
  &\leq \frac{1}{160} \int_{0}^{\tau} \|\nabla\Psi\|^2 d\tau + C (\|\bar{\rho}_{y_1}\|_{L_{y_1}^\infty}^2 + \|d_{1y_1}\|_{L_{y_1}^\infty}^2) \int_{0}^{\tau} \|\phi_{y_1}\|^2 d\tau \\
  &\leq \frac{1}{160} \int_{0}^{\tau} \|\nabla\Psi\|^2 d\tau + C_T (\frac{\varepsilon^2}{\delta^2} + \frac{\varepsilon^4}{\delta^5}) \int_{0}^{\tau} \|\phi_{y_1}\|^2 d\tau.
  \end{align*}
  Similarly, it holds that
  \begin{align*}
  &C\Big|\int_{0}^{\tau}\int_{\bbt_\varepsilon}\int_{\bbr} \frac{2\mu + \lam}{\rho^2}\tilde{\rho}_{y_1y_1}\psi_1\phi_{y_1} dy_1dy_2d\tau\Big| \\
  &\leq \frac{1}{160} \int_{0}^{\tau} \|\phi_{y_1}\|^2 d\tau + C \int_{0}^{\tau}\int_{\bbt_\varepsilon}\int_{\bbr} |\tilde{\rho}_{y_1y_1}\psi_1|^2 dy_1dy_2d\tau \\
  &\leq \frac{1}{160} \int_{0}^{\tau} \|\phi_{y_1}\|^2 d\tau + C \int_{0}^{\tau} \|\tilde{\rho}_{y_1y_1}\|_{L_{y_1}^\infty}^2\|\psi_1\|^2 d\tau \\
  &\leq \frac{1}{160} \int_{0}^{\tau} \|\phi_{y_1}\|^2 d\tau + C_T \varepsilon^3 \sup_{0 \leq t \leq T} \|\tilde{\rho}_{x_1x_1}\|_{L_{x_1}^\infty}^2  \sup_{0 \leq \tau \leq \tau_1(\varepsilon)} \|\psi_1\|^2 \\
  &\leq \frac{1}{160} \int_{0}^{\tau} \|\phi_{y_1}\|^2 d\tau + C_T \varepsilon^3 \sup_{0 \leq t \leq T} (\|\bar{\rho}_{x_1x_1}\|_{L_{x_1}^\infty}^2 + \|d_{1x_1x_1}\|_{L_{x_1}^\infty}^2) \sup_{0 \leq \tau \leq \tau_1(\varepsilon)} \|\psi_1\|^2 \\
  & \leq \frac{1}{160} \int_{0}^{\tau} \|\phi_{y_1}\|^2 d\tau + C_T (\frac{\varepsilon^3}{\delta^4} + \frac{\varepsilon^5}{\delta^7}) \sup_{0 \leq \tau \leq \tau_1(\varepsilon)} \|\psi_1\|^2.
  \end{align*}
  By H\"{o}lder's inequality, Sobolev's inequality and Young's inequality, it holds that
  \begin{align*}
  &C\Big|\int_{0}^{\tau}\int_{\bbt_\varepsilon}\int_{\bbr} \phi \nabla\psi_i \cdot \Psi_{y_i} dy_1dy_2d\tau\Big| \\
  &\leq C \int_{0}^{\tau} \|\nabla\Psi\|\|\phi\|_{L^4}\|\nabla\Psi\|_{L^4} d\tau \leq C \int_{0}^{\tau} \|\nabla\Psi\|\|\phi\|_1\|\nabla\Psi\|_1 d\tau \\
  & \leq C E \int_{0}^{\tau} \|\nabla\Psi\|\|\nabla\Psi\|_1 d\tau \leq \frac{1}{160} \int_{0}^{\tau} \|\nabla\Psi\|^2 d\tau + CE^2 \int_{0}^{\tau} \|\nabla\Psi\|_1^2 d\tau.
  \end{align*}
  By Young's inequality, Lemma \ref{lemma2.2} and Lemma \ref{lemma2.3}, one has
  \begin{align*}
   &C\Big|\int_{0}^{\tau}\int_{\bbt_\varepsilon}\int_{\bbr} \tilde{\rho}_{y_1}\psi_1div\Psi dy_1dy_2d\tau\Big| \\
   &\leq \frac{1}{160} \int_{0}^{\tau} \|\nabla\Psi\|^2 d\tau + C \int_{0}^{\tau}\int_{\bbt_\varepsilon}\int_{\bbr} |\tilde{\rho}_{y_1}\psi_1|^2 dy_1dy_2d\tau \\
   &\leq \frac{1}{160} \int_{0}^{\tau} \|\nabla\Psi\|^2 d\tau + C \int_{0}^{\tau} \|\tilde{\rho}_{y_1}\|_{L_{y_1}^\infty}^2\|\psi_1\|^2 d\tau \\
   &\leq \frac{1}{160} \int_{0}^{\tau} \|\nabla\Psi\|^2 d\tau + C_T \varepsilon \sup_{0 \leq t \leq T} \|\tilde{\rho}_{x_1}\|_{L_{x_1}^\infty}^2  \sup_{0 \leq \tau \leq \tau_1(\varepsilon)} \|\psi_1\|^2 \\
   &\leq \frac{1}{160} \int_{0}^{\tau} \|\nabla\Psi\|^2 d\tau + C_T \varepsilon \sup_{0 \leq t \leq T} (\|\bar{\rho}_{x_1}\|_{L_{x_1}^\infty}^2 + \|d_{1x_1}\|_{L_{x_1}^\infty}^2) \sup_{0 \leq \tau \leq \tau_1(\varepsilon)} \|\psi_1\|^2 \\
   & \leq \frac{1}{160} \int_{0}^{\tau} \|\nabla\Psi\|^2 d\tau + C_T (\frac{\varepsilon}{\delta^2} + \frac{\varepsilon^3}{\delta^5}) \sup_{0 \leq \tau \leq \tau_1(\varepsilon)} \|\psi_1\|^2.
  \end{align*}
  It follows from Young's inequality, Lemma \ref{lemma2.2} and Lemma \ref{lemma2.3} that
  \begin{align*}
  &C\Big|\int_{0}^{\tau}\int_{\bbt_\varepsilon}\int_{\bbr} \frac{2\mu + \lam}{\rho}(\frac{-d_1\bar{u}_1 + d_2}{\tilde{\rho}})_{y_1y_1}\phi_{y_1} dy_1dy_2d\tau\Big| \\
  &\leq \frac{1}{160} \int_{0}^{\tau} \|\phi_{y_1}\|^2 d\tau + C \int_{0}^{\tau}\int_{\bbt_\varepsilon}\int_{\bbr} |(\frac{-d_1\bar{u}_1 + d_2}{\tilde{\rho}})_{y_1y_1}|^2 dy_1dy_2d\tau \\
  &\leq \frac{1}{160} \int_{0}^{\tau} \|\phi_{y_1}\|^2 d\tau + C\varepsilon \int_{0}^{t}\int_{\bbr} |(\frac{-d_1\bar{u}_1 + d_2}{\tilde{\rho}})_{x_1x_1}|^2 dx_1dt \\
  &\leq \frac{1}{160} \int_{0}^{\tau} \|\phi_{y_1}\|^2 d\tau + C_T \frac{\varepsilon^3}{\delta^6},
  \end{align*}
  where we have used the following facts
  \[
  \int_{0}^{t}\int_{\bbr} |(\frac{\bar{u}_1}{\tilde{\rho}}d_{1x_1x_1})|^2 dx_1dt \leq C \int_{0}^{t}\int_{\bbr} |d_{1x_1x_1}|^2 dx_1dt \leq C_T \frac{\varepsilon^2}{\delta^6},
  \]
  \[
  \int_{0}^{t}\int_{\bbr} |(\frac{2d_{1x_1}\bar{u}_{1x_1}}{\tilde{\rho}})|^2 dx_1dt \leq C \int_{0}^{t} \|d_{1x_1}\|_{L_{x_1}^2}^2 \|\bar{u}_{1x_1}\|_{L_{x_1}^{\infty}}^2 dt \leq C_T \frac{\varepsilon^2}{\delta^6},
  \]
  \[
  \int_{0}^{t}\int_{\bbr} |(\frac{d_1\bar{u}_{1x_1x_1}}{\tilde{\rho}})|^2 dx_1dt \leq C \int_{0}^{t} \|d_1\|_{L_{x_1}^2}^2 \|\bar{u}_{1x_1x_1}\|_{L_{x_1}^{\infty}}^2 dt \leq C_T \frac{\varepsilon^2}{\delta^5},
  \]
  \[
  \int_{0}^{t}\int_{\bbr} |(\frac{2d_{1x_1}d_{2x_1}}{\tilde{\rho}^2})|^2 dx_1dt \leq C \int_{0}^{t} \|d_{1x_1}\|_{L_{x_1}^{\infty}}^2 \|d_{2x_1}\|_{L_{x_1}^2}^2 dt \leq C_T \frac{\varepsilon^4}{\delta^9},
  \]
  and
  \[
  \int_{0}^{t}\int_{\bbr} |(\frac{2\bar{u}_1d_1\bar{\rho}_{x_1}^2}{\tilde{\rho}^3})|^2 dx_1dt \leq C \int_{0}^{t} \|d_1\|_{L_{x_1}^2}^2 \|\bar{\rho}_{x_1}\|_{L_{x_1}^{\infty}}^4 dt \leq C_T \frac{\varepsilon^2}{\delta^5}.
  \]
  Similarly, we have
  \begin{align*}
  &C\Big|\int_{0}^{\tau}\int_{\bbt_\varepsilon}\int_{\bbr} \frac{(\tilde{\rho} \tilde{u}_1^2 - \bar{\rho} \bar{u}_1^2 + \bar{u}_1^2d_1 - 2\bar{u}_1d_2)_{y_1}}{\rho}\phi_{y_1} dy_1dy_2d\tau\Big| \\
  &\leq \frac{1}{160} \int_{0}^{\tau} \|\phi_{y_1}\|^2 d\tau + C \int_{0}^{\tau}\int_{\bbt_\varepsilon}\int_{\bbr} |(\tilde{\rho} \tilde{u}_1^2 - \bar{\rho} \bar{u}_1^2 + \bar{u}_1^2d_1 - 2\bar{u}_1d_2)_{y_1}|^2 dy_1dy_2d\tau \\
  &\leq \frac{1}{160} \int_{0}^{\tau} \|\phi_{y_1}\|^2 d\tau + C\varepsilon^{-1} \int_{0}^{t}\int_{\bbr} |(\frac{(d_1\bar{u}_1 - d_2)^2}{\tilde{\rho}})_{x_1}|^2 dx_1dt \\
  &\leq \frac{1}{160} \int_{0}^{\tau} \|\phi_{y_1}\|^2 d\tau + C_T \frac{\varepsilon^3}{\delta^7},
  \end{align*}
  where we have used the following facts
  \[
  \int_{0}^{t}\int_{\bbr} |(\frac{\bar{u}_1^2d_1^2\bar{\rho}_{x_1}}{\tilde{\rho}^2})|^2 dx_1dt \leq C \int_{0}^{t} \|d_1\|_{L_{x_1}^{\infty}}^2 \|\bar{\rho}_{x_1}\|_{L_{x_1}^{\infty}}^2 \|d_1\|_{L_{x_1}^2}^2 dt \leq C_T \frac{\varepsilon^4}{\delta^6},
  \]
  \[
  \int_{0}^{t}\int_{\bbr} |(\frac{\bar{u}_1^2d_1^2d_{1x_1}}{\tilde{\rho}^2})|^2 dx_1dt \leq C \int_{0}^{t} \|d_1\|_{L_{x_1}^{\infty}}^4 \|d_{1x_1}\|_{L_{x_1}^2}^2 dt \leq C_T \frac{\varepsilon^6}{\delta^{10}},
  \]
  and
  \[
  \int_{0}^{t}\int_{\bbr} |(\frac{2\bar{u}_1^2d_1d_{1x_1}}{\tilde{\rho}})|^2 dx_1dt \leq C \int_{0}^{t} \|d_1\|_{L_{x_1}^{\infty}}^2 \|d_{1x_1}\|_{L_{x_1}^2}^2 dt \leq C_T \frac{\varepsilon^4}{\delta^7}.
  \]
  The other terms in \eqref{14} can be analyzed similarly and the details will be omitted for brevity.
  Substituting these estimates into \eqref{14} and taking $\frac{\varepsilon}{\delta^4}$, $\varepsilon$ and $E$ suitably small, we can prove \eqref{LEM4.2} in Lemma \ref{lemma4.2}.
\end{proof}

\begin{lemma} \label{lemma4.3}
	There exists a positive constant $C_T$ such that for $0 \leq \tau \leq \tau_1(\varepsilon)$,
	\begin{align}
	\begin{aligned} \label{LEM4.3}
	&\sup_{0 \leq \tau \leq \tau_1(\varepsilon)} \| (\phi, \Psi)(\tau) \|_1^2 + \int_{0}^{\tau_1(\varepsilon)}\big[ \|\bar{u}_{1y_1}^{1/2} (\phi,  \psi_1) \|^2 + \|(\nabla\phi, \nabla\Psi)\|^2 + \|\nabla^2\Psi\|^2 \big] d\tau \\
	&\leq C_T \frac{\varepsilon}{\delta^4} + C \| (\phi_0, \Psi_0) \|_1^2.
	\end{aligned}
	\end{align}
\end{lemma}
\begin{proof}
  Multiplying the second equation of \eqref{REF} by $-\triangle\Psi/\rho$ gives
  \begin{align}
    \begin{aligned} \label{15}
      &(\frac{|\nabla\Psi|^2}{2})_\tau - div(\psi_{i\tau}\nabla\psi_i + \frac{\mu + \lam}{\rho}div\Psi\nabla div\Psi - \frac{\mu + \lam}{\rho}div\Psi\triangle\Psi) \\
      &\quad + \frac{\mu}{\rho}|\triangle\Psi|^2 + \frac{\mu + \lam}{\rho}|\nabla div\Psi|^2= u_i\Psi_{y_i}\cdot\triangle\Psi + \frac{p'(\rho)}{\rho} \nabla\phi\cdot\triangle\Psi + \tilde{u}_{1y_1}\psi_1\triangle\psi_1 \\
      &\quad+ (\frac{p'(\rho)}{\rho}  - \frac{p'(\tilde{\rho})}{\tilde{\rho}}) \tilde{\rho}_{y_1} \triangle\psi_1
      - \frac{\mu + \lam}{\rho^2}div\Psi\nabla\phi\cdot\triangle\Psi + \frac{\mu + \lam}{\rho^2}div\Psi\nabla\phi\cdot\nabla div\Psi \\
      &\quad - \frac{\mu + \lam}{\rho^2}\tilde{\rho}_{y_1} div\Psi\triangle\psi_1 + \frac{\mu + \lam}{\rho^2}\tilde{\rho}_{y_1} div\Psi div\Psi_{y_1}- \frac{2\mu + \lam}{\rho}(\frac{-d_1\bar{u}_1 + d_2}{\tilde{\rho}})_{y_1y_1}\triangle\psi_1  \\
      &\quad + \frac{(\tilde{\rho} \tilde{u}_1^2 - \bar{\rho} \bar{u}_1^2 + \bar{u}_1^2d_1 - 2\bar{u}_1d_2)_{y_1}}{\rho}\triangle\psi_1 + \frac{(p(\tilde{\rho}) - p(\bar{\rho}) - p'(\bar{\rho})d_1)_{y_1}}{\rho}\triangle\psi_1 \\
      &\quad + \frac{2\mu + \lam}{\tilde{\rho}\rho}\bar{u}_{1y_1y_1}\phi\triangle\psi_1 + \frac{(\tilde{\rho} \tilde{u}_1^2 - \bar{\rho} \bar{u}_1^2 + \bar{u}_1^2d_1 - 2\bar{u}_1d_2)_{y_1}}{\tilde{\rho}\rho}\phi\triangle\psi_1 \\
      &\quad + \frac{(p(\tilde{\rho}) - p(\bar{\rho}) - p'(\bar{\rho})d_1)_{y_1}}{\tilde{\rho}\rho}\phi\triangle\psi_1 := K(\tau, y_1, y_2).
    \end{aligned}
  \end{align}
  Integrating the above equation over $[0, ~\tau]\times\bbr\times\bbt_\varepsilon$ yields
  \begin{align}
    \begin{aligned} \label{16}
      \|\nabla\Psi(\tau)\|^2 + \int_{0}^{\tau} \|\triangle\Psi\|^2 d\tau
      \leq C \|\nabla\Psi_0\|^2 + C |\int_{0}^{\tau}\int_{\bbt_\varepsilon}\int_{\bbr}  K(\tau, y_1, y_2) dy_1dy_2d\tau|.
    \end{aligned}
  \end{align}
  We just estimate some terms on the right-hand side of \eqref{16} in $K$ as follows. 
  It follows from Young's inequality that
  \begin{align*}
  &C\Big|\int_{0}^{\tau}\int_{\bbt_\varepsilon}\int_{\bbr} u_i\Psi_{y_i}\cdot\triangle\Psi + \frac{p'(\rho)}{\rho} \nabla\phi\cdot\triangle\Psi dy_1dy_2d\tau\Big| \\
  &\leq \frac{1}{160} \int_{0}^{\tau} \|\triangle\Psi\|^2 d\tau + C \int_{0}^{\tau} (\|\nabla\phi\|^2 + \|\nabla\Psi\|^2) d\tau.
  \end{align*}
  By Young's inequality, Lemma \ref{lemma2.2} and Lemma \ref{lemma2.3}, one has
  \begin{align*}
  &C\Big|\int_{0}^{\tau}\int_{\bbt_\varepsilon}\int_{\bbr} \tilde{u}_{1y_1}\psi_1\triangle\psi_1 dy_1dy_2d\tau\Big| 
  \leq \frac{1}{160} \int_{0}^{\tau} \|\triangle\psi_1\|^2 d\tau + C \int_{0}^{\tau}\int_{\bbt_\varepsilon}\int_{\bbr} |\tilde{u}_{1y_1}\psi_1|^2 dy_1dy_2d\tau \\
  &\leq \frac{1}{160} \int_{0}^{\tau} \|\triangle\psi_1\|^2 d\tau + C_T (\frac{\varepsilon}{\delta^2} + \frac{\varepsilon^3}{\delta^5}) \sup_{0 \leq \tau \leq \tau_1(\varepsilon)} \|\psi_1\|^2.
  \end{align*}
  By H\"{o}lder's inequality, Sobolev's inequality and Young's inequality, it holds that
  \begin{align*}
  &C\Big|\int_{0}^{\tau}\int_{\bbt_\varepsilon}\int_{\bbr} \frac{\mu + \lam}{\rho^2}div\Psi\nabla\phi\cdot\triangle\Psi dy_1dy_2d\tau\Big|
  \leq C \int_{0}^{\tau} \|\triangle\Psi\|\|\nabla\phi\|_{L^4}\|\nabla\Psi\|_{L^4} d\tau \\
  &\leq C \int_{0}^{\tau} \|\triangle\Psi\|\|\nabla\phi\|_1\|\nabla\Psi\|_1 d\tau
  \leq \frac{1}{160} \int_{0}^{\tau} \|\triangle\Psi\|^2 d\tau + CE^2 \int_{0}^{\tau} \|\nabla\Psi\|_1^2 d\tau.
  \end{align*}
  It follows from Young's inequality, Lemma \ref{lemma2.2} and Lemma \ref{lemma2.3} that
  \begin{align*}
  &C\Big|\int_{0}^{\tau}\int_{\bbt_\varepsilon}\int_{\bbr} \frac{\mu + \lam}{\rho^2}\tilde{\rho}_{y_1} div\Psi\triangle\psi_1 dy_1dy_2d\tau\Big|
  \leq \frac{1}{160} \int_{0}^{\tau} \|\triangle\psi_1\|^2 d\tau + C \int_{0}^{\tau} \|\tilde{\rho}_{y_1}div\Psi\|^2 d\tau \\
  &\leq \frac{1}{160} \int_{0}^{\tau} \|\triangle\psi_1\|^2 d\tau + C_T (\frac{\varepsilon^2}{\delta^2} + \frac{\varepsilon^4}{\delta^5}) \int_{0}^{\tau} \|\nabla\Psi\|^2 d\tau.
  \end{align*}
  All the other terms in $K(\tau, y_1, y_2)$ can be analyzed similarly. Then substituting the resulting estimates into \eqref{16} and the elliptic estimate $\|\triangle\Psi\| \sim \|\nabla^2\Psi\|$ give
  \begin{align}
    \begin{aligned} \label{17}
      &\|\nabla\Psi(\tau)\|^2 + \int_{0}^{\tau} \|\nabla^2\Psi\|^2 d\tau \\
      &\leq C_T \frac{\varepsilon^3}{\delta^7} + C \|\nabla\Psi_0\|^2 + C_T (\frac{\varepsilon}{\delta^2} + \frac{\varepsilon^3}{\delta^5}) \sup_{0 \leq \tau \leq \tau_1(\varepsilon)} \|(\phi, \psi_1)\|^2 \\
      &\quad + C_T (\frac{\varepsilon^2}{\delta^2} + \frac{\varepsilon^4}{\delta^5}) \int_{0}^{\tau} \|\nabla\Psi\|^2 d\tau + C E^2 \int_{0}^{\tau} \|\nabla\Psi\|_1^2 d\tau + C \int_{0}^{\tau} \|(\nabla\phi, \nabla\Psi)\|^2 d\tau.
    \end{aligned}
  \end{align}
  Combining \eqref{LEM4.2} and \eqref{17} and taking $\frac{\varepsilon}{\delta^4}, \varepsilon$ and $E$ suitably small, we complete the proof of Lemma \ref{lemma4.3}.
\end{proof}

\begin{lemma} \label{lemma4.4}
	There exists a positive constant $C_T$ such that for $0 \leq \tau \leq \tau_1(\varepsilon)$,
	\begin{align}
	\begin{aligned} \label{LEM4.4}
	&\sup_{0 \leq \tau \leq \tau_1(\varepsilon)} (\| (\phi, \Psi)(\tau) \|_1^2 + \|\nabla^2\phi(\tau)\|^2) + \int_{0}^{\tau_1(\varepsilon)}\big[ \|\bar{u}_{1y_1}^{1/2} (\phi,  \psi_1) \|^2 + \|(\nabla\phi, \nabla\Psi)\|_1^2 \big] d\tau \\
	&\leq C_T \frac{\varepsilon}{\delta^4} + C (\| (\phi_0, \Psi_0) \|_1^2 + \|\nabla^2\phi_0\|^2) + C E^2 \int_{0}^{\tau_1(\varepsilon)} \|\nabla^3\Psi\|^2 d\tau.
	\end{aligned}
	\end{align}
\end{lemma}
\begin{proof}
  Applying  the operator $\nabla^2$  on the first equation of \eqref{REF} and then multiplying the resulted equation by $\nabla^2\phi/\rho^2$, we have
  \begin{align}
    \begin{aligned} \label{21}
	  &(\frac{|\nabla^2\phi|^2}{2\rho^2})_\tau + div(\frac{\textbf{u} |\nabla^2\phi|^2}{2\rho^2}) + \frac{\nabla^2\phi\cdot\nabla^2 div\Psi}{\rho} \\
	  &= \frac{div\Psi|\nabla^2\phi|^2}{2\rho^2} -\frac{\tilde{\rho}_{y_1y_1}div\Psi\phi_{y_1y_1}}{\rho^2} - \frac{\phi_{y_i} \nabla\phi_{y_i}\cdot\nabla div\Psi}{\rho^2} - \frac{2\tilde{\rho}_{y_1}\nabla\phi_{y_1}\cdot\nabla div\Psi}{\rho^2} \\
	  &\quad - \frac{\nabla\phi\cdot \nabla\phi_{y_i}div\Psi_{y_i}}{\rho^2} - \frac{\nabla\psi_j\cdot\nabla\phi_{y_i}\phi_{y_iy_j}}{\rho^2} - \frac{\psi_{jy_i}\nabla\phi_{y_j}\cdot\nabla\phi_{y_i}}{\rho^2} - \frac{\phi_{y_i}\nabla^2\phi\cdot\nabla^2\psi_i}{\rho^2} \\
	  &\quad - \frac{\tilde{u}_{1y_1y_1}\phi_{y_1}\phi_{y_1y_1}}{\rho^2} - \frac{2\tilde{u}_{1y_1}|\nabla\phi_{y_1}|^2}{\rho^2} - \frac{\tilde{\rho}_{y_1y_1y_1}\psi_1\phi_{y_1y_1}}{\rho^2} - \frac{2\tilde{\rho}_{y_1y_1} \nabla\psi_1\cdot\nabla\phi_{y_1}}{\rho^2} \\
	  &\quad - \frac{\tilde{\rho}_{y_1}\nabla^2\psi_1\cdot\nabla^2\phi}{\rho^2} - \frac{\tilde{u}_{1y_1y_1y_1}\phi\phi_{y_1y_1}}{\rho^2} - \frac{2\tilde{u}_{1y_1y_1} \nabla\phi\cdot\nabla\phi_{y_1}}{\rho^2} + \frac{\tilde{u}_{1y_1} |\nabla^2\phi|^2}{2\rho^2} \\
	  &:= L(\tau, y_1, y_2).
	\end{aligned}
  \end{align}
Then dividing the second equation of \eqref{REF} by $\rho$, applying the operator $\nabla$ on the resulting equation and then multiplying the final equation by $\nabla^2\phi$, we have
  \begin{align}
    \begin{aligned} \label{22}
     &(\nabla\Psi \cdot \nabla^2\phi)_\tau - div(\phi_{y_i\tau}\nabla\psi_i - u_i\phi_{y_j}\nabla\psi_{jy_i} + \textbf{u}\nabla\phi\cdot\triangle\Psi + \frac{\mu}{\rho}\phi_{y_iy_j}\nabla\psi_{iy_j}) \\
     &\quad + (\frac{\mu}{\rho}\nabla\phi_{y_j}\cdot\nabla\psi_{iy_j})_{y_i} +  \frac{p'(\rho)}{\rho}|\nabla^2\phi|^2 - \frac{2\mu + \lam}{\rho}\nabla^2\phi\cdot\nabla^2 div\Psi \\
     &= \rho\nabla div\Psi\cdot\triangle\Psi + \tilde{\rho}_{y_1} div\Psi\triangle\psi_1 + \phi_{y_j}\nabla\psi_i\cdot\nabla\psi_{jy_i} + \tilde{u}_{1y_1}\nabla\phi\cdot\Psi_{y_1y_1} \\
     &\quad + \phi_{y_i}\nabla\psi_i\cdot\triangle\Psi + \tilde{u}_{1y_1}\phi_{y_1}\triangle\psi_1 + \tilde{\rho}_{y_1y_1}\psi_1\triangle\psi_1 + \tilde{\rho}_{y_1}\nabla\psi_1\cdot\triangle\Psi \\
     &\quad  + \tilde{u}_{1y_1y_1}\phi\triangle\psi_1 - \psi_{jy_i}\nabla\psi_i\cdot\nabla\phi_{y_j}- \tilde{u}_{1y_1}\psi_{jy_1}\phi_{y_1y_j} - \tilde{u}_{1y_1y_1}\psi_1\phi_{y_1y_1}  \\
     &\quad- \tilde{u}_{1y_1}\nabla\psi_1\cdot\nabla\phi_{y_1} - (\frac{p'(\rho)}{\rho})'\phi_{y_i}\nabla\phi\cdot\nabla\phi_{y_i} - 2(\frac{p'(\rho)}{\rho})'\tilde{\rho}_{y_1}\nabla\phi\cdot\nabla\phi_{y_1} \\
     &\quad - \big[(\frac{p'(\rho)}{\rho})' - (\frac{p'(\tilde{\rho})}{\tilde{\rho}})'\big]\tilde{\rho}_{y_1}^2\phi_{y_1y_1} - (\frac{p'(\rho)}{\rho}  - \frac{p'(\tilde{\rho})}{\tilde{\rho}}) \tilde{\rho}_{y_1y_1}\phi_{y_1y_1} \\
     &\quad + \frac{\mu}{\rho^2}\nabla\phi\cdot\nabla\psi_{iy_j}\phi_{y_iy_j}- \frac{\mu}{\rho^2}\phi_{y_i}\nabla\phi_{y_j}\cdot\nabla\psi_{iy_j} + \frac{\mu}{\rho^2}\tilde{\rho}_{y_1}\psi_{iy_1y_j}\phi_{y_iy_j} \\
     &\quad  - \frac{\mu}{\rho^2}\tilde{\rho}_{y_1}\nabla\phi_{y_j}\cdot\nabla\psi_{1y_j} - \frac{\mu}{\rho^2}\nabla\phi\cdot\nabla\phi_{y_i}\triangle\psi_i   - \frac{\mu}{\rho^2}\tilde{\rho}_{y_1}\phi_{y_1y_i}\triangle\psi_i\\
     &\quad - \frac{\mu + \lam}{\rho^2}\nabla\phi\cdot\nabla\phi_{y_i}div\Psi_{y_i} - \frac{\mu + \lam}{\rho^2}\tilde{\rho}_{y_1}\phi_{y_1y_i}div\Psi_{y_i}\\
     &\quad+ \frac{2\mu + \lam}{\rho}(\frac{-d_1\bar{u}_1 + d_2}{\tilde{\rho}})_{y_1y_1y_1}\phi_{y_1y_1} - \frac{2\mu + \lam}{\rho^2}(\frac{-d_1\bar{u}_1 + d_2}{\tilde{\rho}})_{y_1y_1}\nabla\phi\cdot\nabla\phi_{y_1} \\
     &\quad - \frac{2\mu + \lam}{\rho^2}(\frac{-d_1\bar{u}_1 + d_2}{\tilde{\rho}})_{y_1y_1}\tilde{\rho}_{y_1}\phi_{y_1y_1} - \frac{2\mu + \lam}{\tilde{\rho}\rho}\bar{u}_{1y_1y_1y_1}\phi\phi_{y_1y_1} \\
     &\quad - \frac{2\mu + \lam}{\rho^2}\bar{u}_{1y_1y_1}\nabla\phi\cdot\nabla\phi_{y_1} + \frac{(2\mu + \lam)(\tilde{\rho} + \rho)}{\tilde{\rho}^2\rho^2}\tilde{\rho}_{y_1}\bar{u}_{1y_1y_1}\phi\phi_{y_1y_1} \\
     &\quad - \big[\frac{(\tilde{\rho} \tilde{u}_1^2 - \bar{\rho} \bar{u}_1^2 + \bar{u}_1^2d_1 - 2\bar{u}_1d_2)_{y_1}}{\tilde{\rho}}\big]_{y_1}\phi_{y_1y_1} - \big[\frac{(p(\tilde{\rho}) - p(\bar{\rho}) - p'(\bar{\rho})d_1)_{y_1}}{\tilde{\rho}}\big]_{y_1}\phi_{y_1y_1} \\
     &:= \rho\nabla div\Psi\cdot\triangle\Psi + M(\tau, y_1, y_2).
    \end{aligned}
  \end{align}
We multiply  \eqref{21} by $2\mu + \lam$, add  the resulting equation and \eqref{22} together, use the same cancellations as in Lemma \ref{lemma4.2},  
and   integrate the final equation over $[0, ~\tau]\times\bbr\times\bbt_\varepsilon$ to get
   \begin{align}
   \begin{aligned}\label{23}
   \int_{\bbt_\varepsilon}\int_{\bbr} \Big(\frac{2\mu + \lam}{2\rho^2}|\nabla^2\phi|^2 + \nabla\Psi\cdot\nabla^2\phi \Big)dy_1dy_2\big|_0^\tau + \int_{0}^{\tau}\int_{\bbt_\varepsilon}\int_{\bbr} \frac{p'(\rho)}{\rho}|\nabla^2\phi|^2 dy_1dy_2d\tau \\
   \di = \int_{0}^{\tau}\int_{\bbt_\varepsilon}\int_{\bbr} \Big[\rho\nabla div\Psi\cdot\triangle\Psi + (2\mu + \lam)L(\tau, y_1, y_2) + M(\tau, y_1, y_2)\Big] dy_1dy_2d\tau.
   \end{aligned}
   \end{align}      
  The combination of \eqref{LEM4.3} and \eqref{23} leads to
    \begin{equation} \label{24}
      \begin{array}{l}
     \di \| (\phi, \Psi)(\tau) \|_1^2 + \|\nabla^2\phi(\tau)\|^2 + \int_{0}^{\tau}\big[ \|\bar{u}_{1y_1}^{1/2} (\phi,  \psi_1) \|^2 + \|(\nabla\phi, \nabla\Psi)\|_1^2 \big] d\tau \\
     \di \leq C_T \frac{\varepsilon}{\delta^4} + C (\| (\phi_0, \Psi_0) \|_1^2 + \|\nabla^2\phi_0\|^2) \\
     \di \quad + C|\int_{0}^{\tau}\int_{\bbt_\varepsilon}\int_{\bbr} \Big[(2\mu + \lam)L(\tau, y_1, y_2) + M(\tau, y_1, y_2)\Big] dy_1dy_2d\tau|.
      \end{array}
    \end{equation}
  Now we estimate some terms on the right-hand side of \eqref{24} selectively. By H\"{o}lder's inequality, Sobolev's inequality and Young's inequality, it holds that
  \begin{align*}
  &C|\int_{0}^{\tau}\int_{\bbt_\varepsilon}\int_{\bbr} \frac{2\mu + \lam}{2\rho^2}div\Psi|\nabla^2\phi|^2 dy_1dy_2d\tau|
  \leq C \int_{0}^{\tau} \|\nabla\Psi\|_{L^\infty} \|\nabla^2\phi\|^2 d\tau \\
  &\leq CE \int_{0}^{\tau} \|\nabla\Psi\|_2 \|\nabla^2\phi\| d\tau
  \leq \frac{1}{160} \int_{0}^{\tau} \|\nabla^2\phi\| d\tau + CE^2 \int_{0}^{\tau} \|\nabla\Psi\|_2^2 d\tau.
  \end{align*}
  By H\"{o}lder's inequality, Sobolev's inequality and Young's inequality, one has
  \begin{align*}
  &C|\int_{0}^{\tau}\int_{\bbt_\varepsilon}\int_{\bbr} \frac{2\mu + \lam}{\rho^2}\phi_{y_i} \nabla\phi_{y_i} \cdot \nabla div\Psi dy_1dy_2d\tau|
  \leq C \int_{0}^{\tau} \|\nabla^2\phi\|\|\nabla\phi\|_{L^4}\|\nabla div\Psi\|_{L^4} d\tau \\
  &\leq C \int_{0}^{\tau} \|\nabla^2\phi\|\|\nabla\phi\|_1\|\nabla^2\Psi\|_1 d\tau
  \leq \frac{1}{160} \int_{0}^{\tau} \|\nabla^2\phi\|^2 d\tau + CE^2 \int_{0}^{\tau} \|\nabla^2\Psi\|_1^2 d\tau.
  \end{align*}
  Similarly, one has
  \begin{align*}
  &C|\int_{0}^{\tau}\int_{\bbt_\varepsilon}\int_{\bbr} (\frac{p'(\rho)}{\rho})'\phi_{y_i}\nabla\phi\cdot\nabla\phi_{y_i} dy_1dy_2d\tau|
  \leq C \int_{0}^{\tau} \|\nabla^2\phi\|\|\nabla\phi\|_{L^4}^2 d\tau \\
  &\leq C \int_{0}^{\tau} \|\nabla^2\phi\|\|\nabla\phi\|_1^2 d\tau
  \leq \frac{1}{160} \int_{0}^{\tau} \|\nabla^2\phi\|^2 d\tau + CE^2 \int_{0}^{\tau} \|\nabla\phi\|_1^2 d\tau.
  \end{align*}
  By Young's inequality, Lemma \ref{lemma2.2} and Lemma \ref{lemma2.3}, it holds that
  \begin{align*}
  &C|\int_{0}^{\tau}\int_{\bbt_\varepsilon}\int_{\bbr} ((\frac{p'(\rho)}{\rho})' - (\frac{p'(\tilde{\rho})}{\tilde{\rho}})')\tilde{\rho}_{y_1}^2\phi_{y_1y_1} dy_1dy_2d\tau| \\
  &\leq \frac{1}{160} \int_{0}^{\tau} \|\phi_{y_1y_1}\|^2 d\tau + C \int_{0}^{\tau}\int_{\bbt_\varepsilon}\int_{\bbr} |\tilde{\rho}_{y_1}^2\phi|^2 dy_1dy_2d\tau \\
  &\leq \frac{1}{160} \int_{0}^{\tau} \|\phi_{y_1y_1}\|^2 d\tau + C \int_{0}^{\tau} \|\tilde{\rho}_{y_1}\|_{L_{y_1}^\infty}^4\|\phi\|^2 d\tau \\
  &\leq \frac{1}{160} \int_{0}^{\tau} \|\phi_{y_1y_1}\|^2 d\tau + C_T \varepsilon^3 \sup_{0 \leq t \leq T} \|\tilde{\rho}_{x_1}\|_{L_{x_1}^\infty}^4  \sup_{0 \leq \tau \leq \tau_1(\varepsilon)} \|\phi\|^2 \\
  &\leq \frac{1}{160} \int_{0}^{\tau} \|\phi_{y_1y_1}\|^2 d\tau + C_T (\frac{\varepsilon^3}{\delta^4} + \frac{\varepsilon^7}{\delta^{10}})  \sup_{0 \leq \tau \leq \tau_1(\varepsilon)} \|\phi\|^2.  
  \end{align*}
  It follows from Young's inequality, Lemma \ref{lemma2.2} and Lemma \ref{lemma2.3} that
  \begin{align*}
  &C|\int_{0}^{\tau}\int_{\bbt_\varepsilon}\int_{\bbr} \frac{2\mu + \lam}{\rho}(\frac{-d_1\bar{u}_1 + d_2}{\tilde{\rho}})_{y_1y_1y_1}\phi_{y_1y_1} dy_1dy_2d\tau| \\
  &\leq \frac{1}{160} \int_{0}^{\tau} \|\phi_{y_1y_1}\|^2 d\tau + C \int_{0}^{\tau}\int_{\bbt_\varepsilon}\int_{\bbr} |(\frac{-d_1\bar{u}_1 + d_2}{\tilde{\rho}})_{y_1y_1y_1}|^2 dy_1dy_2d\tau \\
  &\leq \frac{1}{160} \int_{0}^{\tau} \|\phi_{y_1y_1}\|^2 d\tau + C\varepsilon^3 \int_{0}^{t}\int_{\bbr} |(\frac{-d_1\bar{u}_1 + d_2}{\tilde{\rho}})_{x_1x_1x_1}|^2 dx_1dt \\
  &\leq \frac{1}{160} \int_{0}^{\tau} \|\phi_{y_1y_1}\|^2 d\tau + C_T \frac{\varepsilon^5}{\delta^8},
  \end{align*}
  where we have used the following facts
  \[
  \int_{0}^{t}\int_{\bbr} |(\frac{\bar{u}_1}{\tilde{\rho}}d_{1x_1x_1x_1})|^2 dx_1dt \leq C \int_{0}^{t}\int_{\bbr} |d_{1x_1x_1x_1}|^2 dx_1dt \leq C_T \frac{\varepsilon^2}{\delta^8},
  \]
  \[
  \int_{0}^{t}\int_{\bbr} |(\frac{3d_{1x_1x_1}\bar{u}_{1x_1}}{\tilde{\rho}})|^2 dx_1dt \leq C \int_{0}^{t} \|d_{1x_1x_1}\|_{L_{x_1}^2}^2 \|\bar{u}_{1x_1}\|_{L_{x_1}^{\infty}}^2 dt \leq C_T \frac{\varepsilon^2}{\delta^7},
  \]
  \[
  \int_{0}^{t}\int_{\bbr} |(\frac{3d_{1x_1}\bar{u}_{1x_1x_1}}{\tilde{\rho}})|^2 dx_1dt \leq C \int_{0}^{t} \|d_{1x_1}\|_{L_{x_1}^2}^2 \|\bar{u}_{1x_1x_1}\|_{L_{x_1}^{\infty}}^2 dt \leq C_T \frac{\varepsilon^2}{\delta^7},
  \]
  \[
  \int_{0}^{t}\int_{\bbr} |(\frac{d_1\bar{u}_{1x_1x_1x_1}}{\tilde{\rho}})|^2 dx_1dt \leq C \int_{0}^{t} \|d_1\|_{L_{x_1}^2}^2 \|\bar{u}_{1x_1x_1x_1}\|_{L_{x_1}^{\infty}}^2 dt \leq C_T \frac{\varepsilon^2}{\delta^7},
  \]
  \[
  \int_{0}^{t}\int_{\bbr} |(\frac{3\bar{u}_1d_{1x_1x_1}d_{1x_1}}{\tilde{\rho}^2})|^2 dx_1dt \leq C \int_{0}^{t} \|d_{1x_1}\|_{L_{x_1}^{\infty}}^2 \|d_{1x_1x_1}\|_{L_{x_1}^2}^2 dt \leq C_T \frac{\varepsilon^4}{\delta^{11}},
  \]
  \[
  \int_{0}^{t}\int_{\bbr} |(\frac{6d_{1x_1}\bar{u}_{1x_1}\bar{\rho}_{x_1}}{\tilde{\rho}^2})|^2 dx_1dt \leq C \int_{0}^{t} \|d_{1x_1}\|_{L_{x_1}^2}^2 \|\bar{u}_{1x_1}\|_{L_{x_1}^{\infty}}^4 dt \leq C_T \frac{\varepsilon^2}{\delta^7},
  \]
  \[
  \int_{0}^{t}\int_{\bbr} |(\frac{6d_{1x_1}^2\bar{u}_{1x_1}}{\tilde{\rho}^2})|^2 dx_1dt \leq C \int_{0}^{t} \|\bar{u}_{1x_1}\|_{L_{x_1}^{\infty}}^2 \|d_{1x_1}\|_{L_{x_1}^{\infty}}^2 \|d_{1x_1}\|_{L_{x_1}^2}^2  dt \leq C_T \frac{\varepsilon^4}{\delta^{10}},
  \]
  \[
  \int_{0}^{t}\int_{\bbr} |(\frac{3d_1\bar{u}_{1x_1x_1}\bar{\rho}_{x_1}}{\tilde{\rho}^2})|^2 dx_1dt \leq C \int_{0}^{t} \|d_1\|_{L_{x_1}^2}^2 \|\bar{u}_{1x_1x_1}\|_{L_{x_1}^{\infty}}^2 \|\bar{\rho}_{x_1}\|_{L_{x_1}^{\infty}}^2 dt \leq C_T \frac{\varepsilon^2}{\delta^7},
  \]
  \[
  \int_{0}^{t}\int_{\bbr} |(\frac{6\bar{u}_1d_{1x_1}^3}{\tilde{\rho}^3})|^2 dx_1dt \leq C \int_{0}^{t} \|d_{1x_1}\|_{L_{x_1}^{\infty}}^4 \|d_{1x_1}\|_{L_{x_1}^2}^2 dt \leq C_T \frac{\varepsilon^6}{\delta^{14}},
  \]
  and
  \[
  \int_{0}^{t}\int_{\bbr} |(\frac{6d_1\bar{u}_{1x_1}\bar{\rho}_{x_1}^2}{\tilde{\rho}^3})|^2 dx_1dt \leq C \int_{0}^{t} \|d_1\|_{L_{x_1}^2}^2 \|\bar{u}_{1x_1}\|_{L_{x_1}^{\infty}}^6 dt \leq C_T \frac{\varepsilon^2}{\delta^7}.
  \]
  By Young's inequality, Lemma \ref{lemma2.2} and Lemma \ref{lemma2.3}, one has
  \begin{align*}
  &C |\int_{0}^{\tau}\int_{\bbt_\varepsilon}\int_{\bbr} \frac{2\mu + \lam}{\rho^2}(\frac{-d_1\bar{u}_1 + d_2}{\tilde{\rho}})_{y_1y_1}\nabla\phi\cdot\nabla\phi_{y_1} dy_1dy_2d\tau| \\
  &\leq C \sup_{0 \leq \tau \leq \tau_1(\varepsilon)}\|(\frac{-d_1\bar{u}_1 + d_2}{\tilde{\rho}})_{y_1y_1}\|_{L_{y_1}^\infty} \int_{0}^{\tau} \|\nabla\phi\|\|\nabla\phi_{y_1}\| d\tau \\
  &\leq \frac{1}{160} \int_{0}^{\tau} \|\nabla\phi_{y_1}\|^2 d\tau + C_T (\frac{\varepsilon^6}{\delta^7} + \frac{\varepsilon^8}{\delta^8}) \int_{0}^{\tau} \|\nabla\phi\|^2 d\tau.
  \end{align*}

  It follows from Young's inequality, Lemma \ref{lemma2.2} and Lemma \ref{lemma2.3} that
  \begin{align*}
  &C|\int_{0}^{\tau}\int_{\bbt_\varepsilon}\int_{\bbr} \frac{2\mu + \lam}{\rho^2}(\frac{-d_1\bar{u}_1 + d_2}{\tilde{\rho}})_{y_1y_1}\tilde{\rho}_{y_1}\phi_{y_1y_1} dy_1dy_2d\tau| \\
  &\leq \frac{1}{160} \int_{0}^{\tau} \|\phi_{y_1y_1}\|^2 d\tau + C \int_{0}^{\tau}\int_{\bbt_\varepsilon}\int_{\bbr} |(\frac{-d_1\bar{u}_1 + d_2}{\tilde{\rho}})_{y_1y_1}\tilde{\rho}_{y_1}|^2 dy_1dy_2d\tau \\
  &\leq \frac{1}{160} \int_{0}^{\tau} \|\phi_{y_1y_1}\|^2 d\tau + C (\frac{\varepsilon^2}{\delta^2} + \frac{\varepsilon^4}{\delta^5})\int_{0}^{\tau}\int_{\bbt_\varepsilon}\int_{\bbr} |(\frac{-d_1\bar{u}_1 + d_2}{\tilde{\rho}})_{y_1y_1}|^2 dy_1dy_2d\tau \\
  &\leq \frac{1}{160} \int_{0}^{\tau} \|\phi_{y_1y_1}\|^2 d\tau + C_T \frac{\varepsilon^5}{\delta^8}.
  \end{align*}
  By Young's inequality, Lemma \ref{lemma2.2} and Lemma \ref{lemma2.3}, one has
  \begin{align*}
  &C|\int_{0}^{\tau}\int_{\bbt_\varepsilon}\int_{\bbr} \big[\frac{(\tilde{\rho} \tilde{u}_1^2 - \bar{\rho} \bar{u}_1^2 + \bar{u}_1^2d_1 - 2\bar{u}_1d_2)_{y_1}}{\tilde{\rho}}\big]_{y_1}\phi_{y_1y_1} dy_1dy_2d\tau| \\
  &\leq \frac{1}{160} \int_{0}^{\tau} \|\phi_{y_1y_1}\|^2 d\tau + C \int_{0}^{\tau}\int_{\bbt_\varepsilon}\int_{\bbr} |\big[\frac{(\frac{1}{\tilde{\rho}}(\bar{u}_1d_1 - d_2)^2)_{y_1}}{\tilde{\rho}}\big]_{y_1}|^2 dy_1dy_2d\tau \\
  &\leq \frac{1}{160} \int_{0}^{\tau} \|\phi_{y_1y_1}\|^2 d\tau + C\varepsilon \int_{0}^{t}\int_{\bbr} |\big[\frac{(\frac{1}{\tilde{\rho}}(\bar{u}_1d_1 - d_2)^2)_{x_1}}{\tilde{\rho}}\big]_{x_1}|^2 dx_1dt \\
  &\leq \frac{1}{160} \int_{0}^{\tau} \|\phi_{y_1y_1}\|^2 d\tau + C_T \frac{\varepsilon^5}{\delta^9},
  \end{align*}
   where we have used the facts
  \[
  \int_{0}^{t}\int_{\bbr} |(\frac{3\bar{u}_1^2}{\tilde{\rho}^4}d_1^2\bar{\rho}_{x_1}^2)|^2 dx_1dt \leq C \int_{0}^{t} \|d_1\|_{L_{x_1}^2}^2 \|d_1\|_{L_{x_1}^{\infty}}^2 \|\bar{\rho}_{x_1}\|_{L_{x_1}^{\infty}}^4 dt \leq C_T \frac{\varepsilon^4}{\delta^8},
  \]
  \[
  \int_{0}^{t}\int_{\bbr} |(\frac{3\bar{u}_1^2}{\tilde{\rho}^4}d_1^2d_{1x_1}^2)|^2 dx_1dt \leq C \int_{0}^{t} \|d_1\|_{L_{x_1}^{\infty}}^4 \|d_{1x_1}\|_{L_{x_1}^{\infty}}^2 \|d_{1x_1}\|_{L_{x_1}^2}^2 dt \leq C_T \frac{\varepsilon^8}{\delta^{15}},
  \]
  \[
  \int_{0}^{t}\int_{\bbr} |(\frac{6\bar{u}_1^2}{\tilde{\rho}^3}d_1d_{1x_1}\bar{\rho}_{x_1})|^2 dx_1dt \leq C \int_{0}^{t} \|d_{1x_1}\|_{L_{x_1}^2}^2 \|d_1\|_{L_{x_1}^{\infty}}^2 \|\bar{\rho}_{x_1}\|_{L_{x_1}^{\infty}}^2 dt \leq C_T \frac{\varepsilon^4}{\delta^8},
  \]
  \[
  \int_{0}^{t}\int_{\bbr} |(\frac{6\bar{u}_1^2}{\tilde{\rho}^3}d_1d_{1x_1}^2)|^2 dx_1dt \leq C \int_{0}^{t} \|d_{1x_1}\|_{L_{x_1}^2}^2 \|d_1\|_{L_{x_1}^{\infty}}^2 \|d_{1x_1}\|_{L_{x_1}^{\infty}}^2 dt \leq C_T \frac{\varepsilon^6}{\delta^{12}},
  \]
  \[
  \int_{0}^{t}\int_{\bbr} |(\frac{\bar{u}_1^2}{\tilde{\rho}^3}d_1^2\bar{\rho}_{x_1x_1})|^2 dx_1dt \leq C \int_{0}^{t} \|d_1\|_{L_{x_1}^2}^2 \|d_1\|_{L_{x_1}^{\infty}}^2 \|\bar{\rho}_{x_1x_1}\|_{L_{x_1}^{\infty}}^2 dt \leq C_T \frac{\varepsilon^4}{\delta^8},
  \]
  \[
  \int_{0}^{t}\int_{\bbr} |(\frac{\bar{u}_1^2}{\tilde{\rho}^3}d_1^2d_{1x_1x_1})|^2 dx_1dt \leq C \int_{0}^{t} \|d_1\|_{L_{x_1}^{\infty}}^4 \|d_{1x_1x_1}\|_{L_{x_1}^2}^2 dt \leq C_T \frac{\varepsilon^6}{\delta^{12}},
  \]
  \[
  \int_{0}^{t}\int_{\bbr} |(\frac{2\bar{u}_1^2}{\tilde{\rho}^2}d_{1x_1}^2)|^2 dx_1dt \leq C \int_{0}^{t} \|d_{1x_1}\|_{L_{x_1}^2}^2 \|d_{1x_1}\|_{L_{x_1}^{\infty}}^2 dt \leq C_T \frac{\varepsilon^4}{\delta^9},
  \]
  and
  \[
  \int_{0}^{t}\int_{\bbr} |(\frac{2\bar{u}_1^2}{\tilde{\rho}^2}d_1d_{1x_1x_1})|^2 dx_1dt \leq C \int_{0}^{t} \|d_1\|_{L_{x_1}^{\infty}}^2 \|d_{1x_1x_1}\|_{L_{x_1}^2}^2 dt \leq C_T \frac{\varepsilon^4}{\delta^9}.
  \]
  The other terms in \eqref{24} can be analyzed similarly and the details will be omitted for brevity.
  Substituting these estimates into \eqref{24}, using the elliptic estimates $\|\triangle\Psi\| \sim\|\nabla^2\Psi\|$ and $\|\nabla\triangle\Psi\| \sim \|\nabla^3\Psi\|$ and taking $\frac{\varepsilon}{\delta^4}$, $\varepsilon$ and $E$ suitably small, we can complete the proof of Lemma \ref{lemma4.4}.
\end{proof}

\begin{lemma} \label{lemma4.5}
  There exists a positive constant $C_T$ such that for $0 \leq \tau \leq \tau_1(\varepsilon)$,
  \begin{align}
  \begin{aligned} \label{LEM4.5}
  &\sup_{0 \leq \tau \leq \tau_1(\varepsilon)} \| (\phi, \Psi)(\tau) \|_2^2  + \int_{0}^{\tau_1(\varepsilon)}\big[ \|\bar{u}_{1y_1}^{1/2} (\phi,  \psi_1) \|^2 + \|(\nabla\phi, \nabla\Psi)\|_1^2 + \|\nabla^3\Psi\|^2 \big] d\tau \\
  &\leq C_T \frac{\varepsilon}{\delta^4} + C\| (\phi_0, \Psi_0) \|_2^2.
  \end{aligned}
  \end{align}
\end{lemma}
\begin{proof}
  We   divide the second equation of \eqref{REF} by $\rho$,   apply the operator $\nabla$ to the resulting equation and then multiply the final equation by $-\nabla\triangle\Psi$ to obtain
  \begin{align}
    \begin{aligned} \label{27}
      &(\frac{|\nabla^2\Psi|^2}{2})_\tau - div(\psi_{i\tau y_j}\nabla\psi_{iy_j} + \frac{\mu + \lam}{\rho}div\Psi_{y_i}\nabla div\Psi_{y_i} - \frac{\mu + \lam}{\rho}div\Psi_{y_i}\triangle\Psi_{y_i}) \\
      &\quad + \frac{\mu}{\rho}|\nabla\triangle\Psi|^2 + \frac{\mu + \lam}{\rho}|\nabla^2 div\Psi|^2 \\
      &= u_i\nabla\Psi_{y_i}\cdot\nabla\triangle\Psi + \frac{p'(\rho)}{\rho} \nabla^2\phi\cdot\nabla\triangle\Psi + \psi_{jy_i}\nabla\psi_i\cdot\nabla\triangle\psi_j + \tilde{u}_{1y_1}\Psi_{y_1}\cdot\triangle\Psi_{y_1} \\
      &\quad + \tilde{u}_{1y_1y_1}\psi_1\triangle\psi_{1y_1} + \tilde{u}_{1y_1}\nabla\psi_1\cdot\nabla\triangle\psi_1 + (\frac{p'(\rho)}{\rho})'\phi_{y_i}\nabla\phi\cdot\nabla\triangle\psi_i \\ 
      &\quad + (\frac{p'(\rho)}{\rho})'\tilde{\rho}_{y_1}\nabla\phi\cdot\triangle\Psi_{y_1} + (\frac{p'(\rho)}{\rho})'\tilde{\rho}_{y_1}\nabla\phi\cdot\nabla\triangle\psi_1 + \big[(\frac{p'(\rho)}{\rho})' - (\frac{p'(\tilde{\rho})}{\tilde{\rho}})'\big]\tilde{\rho}_{y_1}^2\triangle\psi_{1y_1} \\ 
      &\quad + (\frac{p'(\rho)}{\rho} - \frac{p'(\tilde{\rho})}{\tilde{\rho}}) \tilde{\rho}_{y_1y_1}\triangle\psi_{1y_1} + \frac{\mu}{\rho^2}\triangle\psi_i\nabla\phi\cdot\nabla\triangle\psi_i + \frac{\mu}{\rho^2}\tilde{\rho}_{y_1}\triangle\Psi\cdot\triangle\Psi_{y_1} \\
      &\quad - \frac{\mu + \lam}{\rho^2}\phi_{y_i}\nabla div\Psi\cdot\nabla\triangle\psi_i - \frac{\mu + \lam}{\rho^2}\tilde{\rho}_{y_1}\nabla div\Psi\cdot\nabla\triangle\psi_1 + \frac{\mu + \lam}{\rho^2}div\Psi_{y_i}\nabla\phi\cdot\nabla div\Psi_{y_i}  \\
      &\quad + \frac{\mu + \lam}{\rho^2}\tilde{\rho}_{y_1}\nabla div\Psi\cdot\nabla div\Psi_{y_1}+ \frac{\mu + \lam}{\rho^2}div\Psi_{y_i}\nabla\phi\cdot\nabla \triangle\psi_i + \frac{\mu + \lam}{\rho^2}\tilde{\rho}_{y_1}\triangle\Psi_{y_1}\cdot\nabla div\Psi \\
      &\quad- \frac{2\mu + \lam}{\rho}(\frac{-d_1\bar{u}_1 + d_2}{\tilde{\rho}})_{y_1y_1y_1}\triangle\psi_{1y_1}  + \frac{2\mu + \lam}{\rho^2}(\frac{-d_1\bar{u}_1 + d_2}{\tilde{\rho}})_{y_1y_1}\nabla\phi\cdot\nabla\triangle\psi_1 \\
      &\quad+ \frac{2\mu + \lam}{\rho^2}(\frac{-d_1\bar{u}_1 + d_2}{\tilde{\rho}})_{y_1y_1}\tilde{\rho}_{y_1}\triangle\psi_{1y_1} + \frac{2\mu + \lam}{\tilde{\rho}\rho}\bar{u}_{1y_1y_1y_1}\phi\triangle\psi_{1y_1} \\
      &\quad + \frac{2\mu + \lam}{\rho^2}\bar{u}_{1y_1y_1}\nabla\phi\cdot\nabla\triangle\psi_1 - \frac{(2\mu + \lam)(\tilde{\rho} + \rho)}{\tilde{\rho}^2\rho^2}\tilde{\rho}_{y_1}\bar{u}_{1y_1y_1}\phi\triangle\psi_{1y_1} \\
      &\quad + \big[\frac{(\tilde{\rho} \tilde{u}_1^2 - \bar{\rho} \bar{u}_1^2 + \bar{u}_1^2d_1 - 2\bar{u}_1d_2)_{y_1}}{\tilde{\rho}}\big]_{y_1}\triangle\psi_{1y_1} + \big[\frac{(p(\tilde{\rho}) - p(\bar{\rho}) - p'(\bar{\rho})d_1)_{y_1}}{\tilde{\rho}}\big]_{y_1}\triangle\psi_{1y_1} \\
      &:= N(\tau, y_1, y_2).
    \end{aligned}
  \end{align}
  Integrating the equation \eqref{27} over $[0, ~\tau]\times\bbr\times\bbt_\varepsilon$ yields that
  \begin{align}
    \begin{aligned} \label{28}
    \|\nabla^2\Psi(\tau)\|^2 + \int_{0}^{\tau} \|\nabla\triangle\Psi\|^2 d\tau
    \leq C \|\nabla^2\Psi_0\|^2 + C |\int_{0}^{\tau}\int_{\bbt_\varepsilon}\int_{\bbr}  N(\tau, y_1, y_2) dy_1dy_2d\tau|.
    \end{aligned}
  \end{align}
  Now we just estimate some terms on the right-hand side of \eqref{28} as follows. 
  It follows from Young's inequality that
  \begin{align*}
  &C |\int_{0}^{\tau}\int_{\bbt_\varepsilon}\int_{\bbr} u_i\nabla\Psi_{y_i}\cdot\nabla\triangle\Psi + \frac{p'(\rho)}{\rho} \nabla^2\phi\cdot\nabla\triangle\Psi dy_1dy_2d\tau| \\
  &\leq \frac{1}{160} \int_{0}^{\tau} \|\nabla\triangle\Psi\|^2 d\tau + C \int_{0}^{\tau} (\|\nabla^2\phi\|^2 + \|\nabla^2\Psi\|^2) d\tau.
  \end{align*}
  By Young's inequality, Lemma \ref{lemma2.2} and Lemma \ref{lemma2.3}, one has
  \begin{align*}
  &C|\int_{0}^{\tau}\int_{\bbt_\varepsilon}\int_{\bbr} (\frac{p'(\rho)}{\rho} -\frac{p'(\tilde{\rho})}{\tilde{\rho}})\tilde{\rho}_{y_1y_1}\triangle\psi_{1y_1} dy_1dy_2d\tau| \\
  &\leq \frac{1}{160} \int_{0}^{\tau} \|\triangle\psi_{1y_1}\|^2 d\tau + C \int_{0}^{\tau}\int_{\bbt_\varepsilon}\int_{\bbr} |\tilde{\rho}_{y_1y_1}\phi|^2 dy_1dy_2d\tau \\
  &\leq \frac{1}{160} \int_{0}^{\tau} \|\triangle\psi_{1y_1}\|^2 d\tau + C_T (\frac{\varepsilon^3}{\delta^4} + \frac{\varepsilon^5}{\delta^7}) \sup_{0 \leq \tau \leq \tau_1(\varepsilon)} \|\phi\|^2.
  \end{align*}
  By H\"{o}lder's inequality, Sobolev's inequality and Young's inequality, it holds that
  \begin{align*}
  &C|\int_{0}^{\tau}\int_{\bbt_\varepsilon}\int_{\bbr} \frac{\mu}{\rho^2}\triangle\psi_i\nabla\phi\cdot\nabla\triangle\psi_i dy_1dy_2d\tau|
  \leq C \int_{0}^{\tau} \|\nabla\triangle\Psi\|\|\triangle\Psi\|_{L^4}\|\nabla\phi\|_{L^4} d\tau \\
  &\leq C \int_{0}^{\tau} \|\nabla\triangle\Psi\|\|\triangle\Psi\|_1\|\nabla\phi\|_1 d\tau
  \leq \frac{1}{160} \int_{0}^{\tau} \|\nabla\triangle\Psi\|^2 d\tau + CE^2 \int_{0}^{\tau} \|\triangle\Psi\|_1^2 d\tau.
  \end{align*}
  By Young's inequality, Lemma \ref{lemma2.2} and Lemma \ref{lemma2.3}, one has
  \begin{align*}
  &C |\int_{0}^{\tau}\int_{\bbt_\varepsilon}\int_{\bbr} \frac{\mu}{\rho^2}\tilde{\rho}_{y_1}\triangle\Psi\cdot\triangle\Psi_{y_1} dy_1dy_2d\tau| \\
  &\leq \frac{1}{160} \int_{0}^{\tau} \|\triangle\Psi_{y_1}\|^2 d\tau + C \int_{0}^{\tau} \|\tilde{\rho}_{y_1}\triangle\Psi\|^2 d\tau \\
  &\leq \frac{1}{160} \int_{0}^{\tau} \|\triangle\Psi_{y_1}\|^2 d\tau + C \int_{0}^{\tau} \|\tilde{\rho}_{y_1}\|_{L_{y_1}^\infty}^2 \|\triangle\Psi\|^2 d\tau \\
  &\leq \frac{1}{160} \int_{0}^{\tau} \|\triangle\Psi_{y_1}\|^2 d\tau + C_T (\frac{\varepsilon^2}{\delta^2} + \frac{\varepsilon^4}{\delta^5}) \int_{0}^{\tau} \|\triangle\Psi\|^2 d\tau.
  \end{align*}
  By Young's inequality, Lemma \ref{lemma2.2} and Lemma \ref{lemma2.3}, it holds that
  \begin{align*}
  &C |\int_{0}^{\tau}\int_{\bbt_\varepsilon}\int_{\bbr} \frac{2\mu + \lam}{\rho^2}(\frac{-d_1\bar{u}_1 + d_2}{\tilde{\rho}})_{y_1y_1}\nabla\phi\cdot\nabla\triangle\psi_1 dy_1dy_2d\tau| \\
  &\leq C \int_{0}^{\tau} \|(\frac{-d_1\bar{u}_1 + d_2}{\tilde{\rho}})_{y_1y_1}\|_{L_{y_1}^\infty} \|\nabla\phi\|\|\nabla\triangle\psi_1\| d\tau \\
  &\leq \frac{1}{160} \int_{0}^{\tau} \|\nabla\triangle\psi_1\|^2 d\tau + C_T (\frac{\varepsilon^6}{\delta^7} + \frac{\varepsilon^8}{\delta^8}) \int_{0}^{\tau} \|\nabla\phi\|^2 d\tau.
  \end{align*}
  Similarly, one has
  \begin{align*}
  &C|\int_{0}^{\tau}\int_{\bbt_\varepsilon}\int_{\bbr} \big[\frac{(\tilde{\rho} \tilde{u}_1^2 - \bar{\rho} \bar{u}_1^2 + \bar{u}_1^2d_1 - 2\bar{u}_1d_2)_{y_1}}{\tilde{\rho}}\big]_{y_1}\triangle\psi_{1y_1} dy_1dy_2d\tau| \\
  &\leq \frac{1}{160} \int_{0}^{\tau} \|\triangle\psi_{1y_1}\|^2 d\tau + C \int_{0}^{\tau}\int_{\bbt_\varepsilon}\int_{\bbr} |\big[\frac{(\frac{1}{\tilde{\rho}}(\bar{u}_1d_1 - d_2)^2)_{y_1}}{\tilde{\rho}}\big]_{y_1}|^2 dy_1dy_2d\tau \\
  &\leq \frac{1}{160} \int_{0}^{\tau} \|\triangle\psi_{1y_1}\|^2 d\tau + C_T \frac{\varepsilon^5}{\delta^9} + C_T \frac{\varepsilon^9}{\delta^{15}} + C_T \frac{\varepsilon^7}{\delta^{12}}.
  \end{align*}
 The other terms in $N(\tau, y_1, y_2)$ can be estimated similarly as before. Then substituting all the estimates into \eqref{28} and using the standard elliptic estimates $\|\triangle\Psi\| \sim \|\nabla^2\Psi\|$ and $\|\nabla\triangle\Psi\| \sim \|\nabla^3\Psi\|$, it holds that
  \begin{equation}\label{29}
  \begin{array}{ll}
  \di \|\nabla^2\Psi(\tau)\|^2 + \int_{0}^{\tau} \|\nabla^3\Psi\|^2 d\tau \\
  \di \leq C_T \frac{\varepsilon^5}{\delta^9} + C \|\nabla^2\Psi_0\|^2 + C \int_{0}^{\tau} \|(\nabla^2\phi, \nabla^2\Psi)\|^2 d\tau + C E^2 \int_{0}^{\tau} \|(\nabla \phi, \nabla \Psi, \nabla^2\Psi)\|_1^2 d\tau \\
  \di \quad + C_T (\frac{\varepsilon^3}{\delta^4} + \frac{\varepsilon^5}{\delta^7}) \sup_{0 \leq \tau \leq \tau_1(\varepsilon)} \|(\phi, \psi_1)\|^2 
  + C_T (\frac{\varepsilon^2}{\delta^2} + \frac{\varepsilon^4}{\delta^5}) \int_{0}^{\tau} \|(\nabla\phi, \nabla\Psi, \nabla^2\Psi)\|^2 d\tau.
  \end{array}
  \end{equation}

  Combining \eqref{LEM4.4} and \eqref{29} and taking $\frac{\varepsilon}{\delta^4}, \varepsilon$ and $E$ suitably small, we complete the proof of Lemma \ref{lemma4.5}.
\end{proof}

Finally,    taking $\frac{\varepsilon}{\delta^4}, \varepsilon$ and $E$ sufficiently small, saying $\varepsilon \leq \varepsilon_1(\varepsilon), \delta = \varepsilon^{1/6}, E \ll 1$, we obtain the desired a priori estimate \eqref{PRO3.2}, which finishes the proof of the a priori estimates in Proposition \ref{proposition3.2}. Therefore the Theorem \ref{theorem1} is proved.

 \bigskip
   
 \section*{Acknowledgement} 
 The research  of D. Wang is partially supported by the
National Science Foundation under grants DMS-1312800 and DMS-1613213. 
Y. Wang is supported by NSFC grants No. 11671385 and 11688101 and CAS Interdisciplinary Innovation Team.

 \bigskip


\end{document}